\documentclass[a4paper,12pt]{article}

\usepackage{amsmath}
\usepackage{amsfonts}
\usepackage{amsthm}
\usepackage{amssymb}
\usepackage{graphicx}
\usepackage{tikz}
\usepackage{latexsym}
\usepackage{mathrsfs}
\usepackage[applemac]{inputenc}
\usetikzlibrary{positioning}
\usetikzlibrary{matrix}
\usetikzlibrary{arrows}
\usetikzlibrary{decorations.markings,positioning}
\usepackage{version}
\usepackage{bbm} 
\usepackage[square, numbers]{natbib}

\newtheorem{thm}[equation]{Theorem}
\newtheorem{cor}[equation]{Corollary}

\newtheorem{prop}[equation]{Proposition}

      \newtheorem{defi}[equation]{Definition}

      \newtheorem{rem}[equation]{Remark}

      \newtheorem{ex}[equation]{Example}

      \numberwithin{equation}{subsection}

\usepackage{geometry}
\title{\bf Simplicial Nerve of an \(\mathcal{A}_∞\)-category.}

\author{Giovanni Faonte \footnote{Department of Mathematics, Yale University, 10 Hillhouse Avenue, New Haven CT 06520 USA. Email: giovanni.faonte@yale.edu}
  }

\begin{document}


\maketitle

\abstract{In this paper we introduce a functor, called the simplicial nerve of an $\mathcal{A}_{\infty}$-category, defined on the category of (small) $\mathcal{A}_{\infty}$-categories with values in simplicial sets. We prove that the nerve of an $\mathcal{A}_{\infty}$-category is an $\infty$-category in the sense of J.Lurie \cite{Lurie 1}. This construction generalize the nerve construction for differential graded categories given in \cite{Lurie 2}. We prove that if a differential graded category is pretriangulated in the sense of A.I.Bondal-M.Kapranov \cite{MK}, then its nerve is a stable $\infty$-category in the sense of J.Lurie \cite{Lurie 2}.}

\tableofcontents

\newpage
\section{Introduction}

$\mathcal{A}_{\infty}$-algebras were introduced by J.D.Stasheff \cite{Sta} in order to encode the notion of an operation associative up to a coherent system of homotopies. An $\mathcal{A}_{\infty}$-algebra is a $\mathbb{Z}$-graded vector space $V$ over a base field $\mathbb{K}$, togheter with degree $2-n$ graded maps
\[
m_n:V^{\otimes n}\to V
\]
for $n\ge 1$, such that 
\[
\sum_{n=r+t+s} (-1)^{sr+t} m_{r+t+1}(Id^{\otimes^r}\otimes m_s \otimes Id^{\otimes^t})=0
\]
$\mathcal{A}_{\infty}$-algebras with $m_n=0$, for $n>2$, are differential graded algebras. 

M.Kontsevich-Y.Soibelman \cite{Ko} gave a geometric interpretation of the notion of $\mathcal{A}_{\infty}$-algebra. Namely, $\mathcal{A}_{\infty}$-algebras correspond to non-commutative formal graded manifolds over a field $\mathbb{K}$, together with a marked $\mathbb{K}$-point and a vector field of degree $+1$, such that $d|_{pt}=0$ and $[d,d]=0$. Such vector fields are called homological vector fields. In this interpretation, the maps $m_n$ are the coefficients of the Taylor expansion of the vector field $d$. 

$\mathcal{A}_∞$-categories were introduced by K.Fukaya \cite{Fuk} as a generalization of the notion of $\mathcal{A}_∞$-algebra. Indeed, $\mathcal{A}_∞$-algebras correspond to $\mathcal{A}_∞$-categories with one object. The work of Fukaya exploited the connection between $\mathcal{A}_∞$-categories and symplectic geometry. Namely, the Fukaya category $\mathcal{F}(X)$ is an $\mathcal{A}_∞$-category associated to a symplectic manifold $X$ and plays a relevant role in mirror symmetry. On the geometric side, $\mathcal{A}_∞$-categories correspond to non-commutative formal graded manifolds together with the choice of a closed subscheme of disjoint points (objects) and an homological vector field compatible with the embedding of this subscheme and with the projection onto it. Similarly to the case of $\mathcal{A}_{\infty}$-algebras, differential graded categories have a faithful embedding in the category of $\mathcal{A}_∞$-categories.

In Section 2 we define a functor, called the simplicial nerve of an $\mathcal{A}_{\infty}$-category
\[
N_{\mathcal{A}_∞}: \mathcal{A}_∞Cat_{\mathbb{K}}\to SSet
\]
and we prove that the nerve of an $\mathcal{A}_{\infty}$-category is and $\infty$-category in the sense of J.Lurie \cite{Lurie 1}. The nerve construction is performed defining a cosimplicial $\mathcal{A}_{\infty}$-category, generated by the standard simplicies and denoted by  $\mathcal{A}[\Delta^{-}]$, and by taking $\mathcal{A}_{\infty}$-functors from $\mathcal{A}[\Delta^{-}]$ to any $\mathcal{A}_{\infty}$-category $\mathcal{A}$. Its restriction to the category of (small) dg-categories, that we call the small dg-nerve of a dg-category, provides a functorial description of the differential graded nerve defined by J.Lurie in \cite{Lurie 2}
\[
N_{dg}^{sm}: dgCat_{\mathbb{K}}\to SSet
\]
The relevant fact about this construction is that it identifies the homotopy category of $N_{dg}^{sm}(\mathcal{D})$ and the $0$-th cohomology category of $\mathcal{D}$. We remark that the small dg-nerve was earlier defined by V. A. Hinich and V. V. Schechtman in \cite{HSc}.

In fact, there are two different nerve-type constructions associating a simplicial set to a dg-category. J.Lurie in \cite{Lurie 2} constructed also what we will call the big dg-nerve of a dg-category, to underline the difference with the small dg-nerve. Given a differential graded category $\mathcal{D}$, consider the simplicial category $\mathcal{D}_{\Delta}$ whose mapping spaces are obtained by applying the Dold-Kan correspondence to the truncation of the cochain complex of morphisms in $\mathcal{D}$. General fact is that simplicial categories have themselves a nerve construction. More precisely, there is an adjunction \citep[Chap. 2, \S 2.2]{Lurie 1}
\[
\begin{tikzpicture}
\matrix (m) [matrix of math nodes, row sep=2em, column sep=3pc,
  text width=2.2pc, text height=1pc, text depth=.5pc] { 
   SCat & SSet \\
  }; 
\path[<-] 
(m-1-1.15) edge node[above] {\tiny $\mathcal{C}[-]$}  (m-1-2.165)
(m-1-2.-165) edge node[below] {\tiny $ N_{SCat}$}  (m-1-1.-15);
\end{tikzpicture}
    \]
and the big dg-nerve is defined as the value of $N_{SCat}$ on the simplicial category $\mathcal{D}_{\Delta}$ and denoted by $N_{dg}^{big}(\mathcal{D})$. In section 3 we compare those two $\infty$-categories associated to a differential graded category. Namely, $N_{dg}^{big}(\mathcal{D})$ and $N_{dg}^{sm}(\mathcal{D})$ are equivalent $\infty$-categories \cite{Lurie 2}. We provide the construction of such equivalence describing a cubical interpretation of the mapping spaces of the simplicial categories generated by the standard $n$-simplex $\mathcal{C}[\Delta^n]$.

In Section 4 we establish a connection between pretriangulated differential graded categories in the sense of A.I.Bondal-M.Kapranov \cite{MK} and stable $\infty$-categories in the sense of J.Lurie \cite{Lurie 2}. Namely, we prove that if $\mathcal{D}$ is a pretriangulated differential graded category, then $N_{dg}^{big}(\mathcal{D})$, and hence $N_{dg}^{sm}(\mathcal{D})$, is a stable $\infty$-category. Moreover, $H^0(\mathcal{D})$ is identified, as a triangulated category, with the homotopy category of $N_{dg}^{big}(\mathcal{D})$.
It is known that the notion of triangulated category introduced by J.-L.Verdier \cite{Ver} lacks functoriality for cones. For this purpose, A.I.Bondal-M.Kapranov \cite{MK} introduced the notion of a pretriangulated differential graded category. To a dg-category $\mathcal{D}$ it is possible to associate the dg-category of twisted complexes in $\mathcal{D}$, denoted by $PreTr(\mathcal{D})$, whose construction has to be understood as a triangulated hull of $\mathcal{D}$. $PreTr(\mathcal{D})$ has a shift functor and a functorial notion of cones inducing a triangulated structure on its $0$-th cohomology category $H^0(PreTr(\mathcal{D}))$. In particular, if $\mathcal{D}$ is pretriangulated, the dg-embedding $\mathcal{D}\to PreTr(\mathcal{D})$ is a quasi-equivalence of dg-categories and hence $\mathcal{D}$ inherits shift and cones from $PreTr(\mathcal{D})$ making $H^0(\mathcal{D})$ into a triangulated category. A similar notion, on the $\infty$-categorical side, is given by the notion of stable $\infty$-category. Stable $\infty$-categories, introduced by J.Lurie in \cite{Lurie 2}, are an axiomatization of the property of stable homotopy theory. In this context, it is possible to define loop and suspension $\infty$-functors and to give meaning to the notion of exact sequences in a way that the homotopy category inherits a triangulated structure. Our result allows hence to compare the two notions. 
\vspace{5 mm}

Finally, let us discuss the relations of this paper with earlier works. As already pointed, both nerve constructions for dg-categories and their main properties are stated by J.Lurie in \citep[Chap. 1, \S 1.3.1]{Lurie 2}. We extend $N_{dg}^{sm}$, functorially, to the context of $\mathcal{A}_{\infty}$-categories, thus realizing a suggestion in \citep[Remark 1.3.1.7]{Lurie 2} and we provide more details on how to construct the equivalence between $N_{dg}^{sm}$ and $N_{dg}^{big}$. In a recent paper \cite{Coh} L.Cohn proves that there exists an equivalence of $\infty$-categories between: 
\begin{itemize}
\item the underlying $\infty$-category associated to the Morita model category structure on $dgCat_{\mathbb{K}}$ \cite{Toe2}, for which fibrant objects are idempotent complete pretriangulated dg-categories.
\item the $\infty$-category of stable idempotent complete $\infty$-categories enriched over the Eilenberg-MacLane spectra $H\mathbb{K}$.
\end{itemize}
This result establishes a connection between pretruangulated dg-categories and stable $\infty$-categories which is similar to ours but with idempotent completion forced on both sides. However, neither Verider's notion of a triangulated category, nor Bondal-Kapranov's notion of a pretriangulated dg-category, nor J.Lurie's notion of stable $\infty$-category requires idempotent completion. From this perspective, our approach relates the original notions and the desirability of such a relation was emphasized in \citep[Remark 1.1.0.2]{Lurie 2}. Moreover, our correspondence is explicit: given a pretriangulated dg-category, idempotent complete or not, we prove directly that the associated nerve is a stable $\infty$-category.
\vspace{5 mm}

The author would like to thank T.Dyckerhoff and M.Kapranov for useful discussions and precious insights given through the realization of this work.

\newpage

\section{The nerve construction for $\mathcal{A}_∞$-categories}

\subsection{$\mathcal{A}_∞$-categories and dg-categories}

Let $\mathbb{K}$ be a field. Consider the category $Vect_{\mathbb{Z}}(\mathbb{K})$ whose objects are $\mathbb{Z}$-graded vector spaces over $\mathbb{K}$
\[
V=\bigoplus_{p\in \mathbb{Z}} V^p
\]
and morphisms are given by degree preserving $\mathbb{K}$-linear maps. The category $Vect_{\mathbb{Z}}(\mathbb{K})$ has a symmetric monoidal structure given by the tensor product of graded vector spaces
\[
(V\otimes W)^n=\bigoplus_{p+q=n}V^p\otimes W^q
\]
and symmetric structure given by the natural isomorphism
\[
sym(V,W):V\otimes W \to W\otimes V
\] 
defined on homogeneous elements by
\[
sym(V,W)(v\otimes w)=(-1)^{deg(v)deg(w)}w\otimes v
\]
Given $V,W \in Vect_{\mathbb{Z}}(\mathbb{K})$, the graded space of morphisms $Hom^{\bullet}_{Vect_{\mathbb{Z}}(\mathbb{K})}(V,W)$ is
\[
Hom^{n}_{Vect_{\mathbb{Z}}(\mathbb{K})}(V,W)=\prod_{p\in \mathbb{Z}} Hom_{Vect(\mathbb{K})}(V^p,W^{p+n})
\]
A map $f\in Hom^{n}_{Vect_{\mathbb{Z}}(\mathbb{K})}(V,W)$ is called of degree $n$. \\
Given two morphisms $f\in Hom^{n}_{Vect_{\mathbb{Z}}(\mathbb{K})}(V,W)$ and $g\in Hom^{m}_{Vect_{\mathbb{Z}}(\mathbb{K})}(V',W')$, their tensor product is the map $f\otimes g\in Hom^{n+m}_{Vect_{\mathbb{Z}}(\mathbb{K})}(V\otimes W,V'\otimes W')$, defined on homogeneous elements by
\[
(f\otimes g)(x\otimes y)=(-1)^{deg(x)deg(g)}f(x)\otimes g(y)
\]
Given $V\in Vect_{\mathbb{Z}}(\mathbb{K})$, its suspension $s(V)$ is the graded vector space
\[
s(V)=V\otimes \mathbb{K}[1]
\]
where $\mathbb{K}[1]$ is the graded vector space generated over $\mathbb{K}$ by one generator $\{\mathbbm{1}\}$ of degree $-1$. The suspension map is the invertible morphism of degree $-1$ 
\[
s:V\to s(V)
\]
defined by 
\[
s(v)=v\otimes \mathbbm{1}
\]
Given $V\in Vect_{\mathbb{Z}}(\mathbb{K})$, its reduced tensor coalgebra is the graded vector space 
\[
\overline{T}(V)=\bigoplus_{n\ge 1} V^{\otimes n}
\]
with coalgebra structure given by the comultiplication
\[
\Delta(v_1\otimes \dots \otimes v_n)=\sum_{0\le i\le n} (v_1\otimes \dots \otimes v_i)\otimes (v_{i+1}\otimes \dots \otimes v_n)
\]
We recall the notion of an $\mathcal{A}_∞$-algebra structure on a graded vector space $V$. 
\begin{defi}[$\mathcal{A}_∞$-algebra]
Given $V\in Vect_{\mathbb{Z}}(\mathbb{K})$, an $\mathcal{A}_∞$-algebra structure on $V$ is the datum of a coderivation of degree $+1$ on the reduced tensor coalgebra $\overline{T}(s(V))$
\[
b:\overline{T}(s(V))\to \overline{T}(s(V))
\]
such that $b^2=0$.
\end{defi}
\begin{rem}\normalfont
A coderivation $b$ of degree $+1$ on $\overline{T}(s(V))$ is determined by graded morphisms of degree $+1$
\[
b_i:s(V)^{\otimes i}\to s(V)
\]
for $i\ge 1$. Moreover the equation $b^2=0$ translates into the system of equations, for every $n\ge 1$ 
\begin{equation}\label{a1}
\sum_{n=r+t+s} b_{r+t+1}(Id^{\otimes^r}\otimes b_s \otimes Id^{\otimes^t})=0
\end{equation}
Consider the sequence of graded morphisms of degree $2-i$
\[
m_i:V^{\otimes i}\to V
\]
for $i\ge 1$ defined by the condition
\[
b_i=s\circ m_i\circ (s^{-1})^{\otimes i}
\]
then the equation (\ref{a1}) can be written in terms of $m_i$'s as
\begin{equation}\label{a2}
\sum_{n=r+t+s} (-1)^{sr+t} m_{r+t+1}(Id^{\otimes^r}\otimes m_s \otimes Id^{\otimes^t})=0
\end{equation}
Note that there are several ways of associating the morphisms $m_i$ to $b$. To those corresponds different signs convention in the equation (\ref{a2}). The sign convention adopted in this paper is similar to the one adopted by K.Lefèvre-Hasegawa in \cite{Has}. 
 \end{rem}
\begin{defi}[Strictly Unital $\mathcal{A}_∞$-algebra]
A strictly unital $\mathcal{A}_∞$-algebra is an $\mathcal{A}_∞$-algebra $V$ together with the choice of an element $1_V\in V$ of degree $0$, called strict unit, such that
\[
m_2(1_V\otimes v)=v=m_2(v\otimes 1_V)
\]
and for $n>2$, $1\le j\le n$
\[
m_n(v_1\otimes \dots \otimes v_{j-1}\otimes 1_V\otimes v_{j+1}\otimes \dots v_n)=0
\]
\end{defi}
It is easy to check that the choice of a strict unit for an $\mathcal{A}_∞$-algebra is unique.
\begin{defi}[Morphism of $\mathcal{A}_∞$-algebras]
Given two $\mathcal{A}_∞$-algebras $(V,b)$ and $(W,b')$, a morphism of $\mathcal{A}_∞$-algebras is a degree $0$ morphism of graded coalgebras
\[
h:\overline{T}(s(V))\to \overline{T}(s(W))
\]
such that
\[
b'\circ h=h\circ b
\]
\end{defi}
\begin{rem}\normalfont \label{rem1}
A morphism of $\mathcal{A}_∞$-algebras is determined by a sequence of graded morphisms of degree $0$
\[
h_i:s(V)^{\otimes i}\to s(W)
\]
for $i\ge 1$ and the equation $b'\circ h=h\circ b$ translate in the system of equations for $n\ge 1$
\[
\sum_{n=r+t+s} h_{r+t+1}(Id^{\otimes^r}\otimes b_s \otimes Id^{\otimes^t})=\sum_{\substack{1\le r\le n\\ i_1+\dots +i_r=n}} b'_{r}(h_{i_1}\otimes \dots \otimes h_{i_r})
\]
Define the sequence of graded morphisms of degree $1-i$
\[
f_i:V^{\otimes i}\to W
\]
for $i\ge 1$ by
\[
h_i=s'\circ f_i\circ (s^{-1})^{\otimes i}
\]
then the system of equations defined by (\ref{a2}) translates into the system of equations for $n\ge 1$
\[
\sum_{n=r+t+s} (-1)^{sr+t} f_{r+t+1}(Id^{\otimes^r}\otimes m_s \otimes Id^{\otimes^t})=\sum_{\substack{1\le r\le n\\ i_1+\dots +i_r=n}} (-1)^{\epsilon_r(i_1,\dots ,i_r)} m'_{r}(f_{i_1}\otimes \dots \otimes f_{i_r})
\]
where
\[
\epsilon_r(i_1,\dots ,i_r)=\sum_{2\le k\le r}\left( (1-i_k)\sum_{1\le l\le k-1} i_l \right)
\]
There is a strictly associative and unital composition of morphisms of  $\mathcal{A}_∞$-algebras. Namely, given $h:(V,b)\to (W,b')$ and $k:(W,b')\to (Z,b'')$ morphisms of $\mathcal{A}_∞$-algebras, their composition $k\circ h$ is the composition as graded morphisms. Moreover, if $f_i:V^{\otimes i}\to W$ are the graded morphisms associated to $h$ and $g_i:W^{\otimes i}\to Z$ associated to $k$, the composition $k\circ h$ has associated graded morphisms given by 
\[
e_k=\sum_{r=1}^k \sum_{i_1+\dots +i_r=k}(-1)^{\epsilon_r(i_1,\dots ,i_r)} g_{r}(f_{i_1}\otimes \dots \otimes f_{i_r})
\]
The unit is given by $Id_{V}$. Those data define the category $\mathcal{A}_{\infty}Alg(\mathbb{K})$ of $\mathcal{A}_{\infty}$-algebras over $\mathbb{K}$. 
\end{rem}
\begin{defi}[Morphism of strictly unital $\mathcal{A}_∞$-algebras]
A morphism of strictly unital $\mathcal{A}_∞$-algebras is a morphism of $\mathcal{A}_∞$-algebras
\[
f_i:V^{\otimes i}\to W
\]
such that
\[
f_1(1_V)=1_W
\]
and for $n>1$, $1\le j\le n$
\[
f_n(v_1\otimes \dots \otimes v_{j-1}\otimes 1_V\otimes v_{j+1}\otimes \dots v_n)=0
\]
\end{defi}
It is straightforward to check that strictly unitality is preserved under composition of morphisms of $\mathcal{A}_∞$-algebras and that $Id_{V}$  is strictly uniltal. 
\begin{ex}\normalfont
Recall that a differential graded algebra (dg-algebra) over a field $\mathbb{K}$ is a graded vector space $V$ over $\mathbb{K}$ endowed with a graded morphism of degree $0$, called product
\[
m: V\otimes V\to V
\]
and a graded morphism of degree $+1$, called differential
\[
d:V\to V
\]
such that 
\begin{align*}
&d^2=0\\
&d(v_1\cdot v_2)=d(v_1)\cdot v_2 + (-1)^{deg(v_1)}v_1\cdot d(v_2) \
\end{align*}
where $v_1\cdot v_2=m(v_1,v_2)$. Given a dg-algebra $V$, its cohomology algebra is the graded algebra
\[
H^{\star}(V)=\frac{Ker(d)}{Im(d)}
\]
with product induced by the product in $V$. A morphism of dg-algebras is a graded morphism of degree $0$
\[
f:V\to W
\]
such that
\begin{align*}
&d\circ f=f\circ d\\
&f(v_1\cdot v_2)=f(v_1)\cdot f(v_2) \,
\end{align*}
Composition of morphisms of dg-algebras is composition as graded morphisms. This composition defines the category $dgAlg(\mathbb{K})$ of dg-algebras over a field $\mathbb{K}$. Moreover, there is a faithful functor
\[
i:dgAlg(\mathbb{K})\to \mathcal{A}_{\infty}Alg(\mathbb{K})
\]
sending a dg-algebra $V$ in the $\mathcal{A}_{\infty}$-algebra with the same underlying graded vector space $V$ and with
\[
m_1=d,
m_2=m,
m_n=0, \quad \text{for $n>2$} 
\]
and a morphism of dg-algebras 
\[
f:V\to W
\]
in the morphism of  $\mathcal{A}_{\infty}$-algebras 
\[
i(f):i(V)\to i(W)
\]
with
\[
i(f)_1=f,
i(f)_n=0, \quad \text{for $n>1$} \
\]
A similar statement holds in the unital case.
\end{ex}
\newpage
We recall now the notion of $\mathcal{A}_∞$-category and of $\mathcal{A}_∞$-functor. 
\begin{defi}[$\mathcal{A}_∞$-category]
Let $\mathbb{K}$ be a field. An $\mathcal{A}_{\infty}$-category $\mathcal{A}$ is the data of:
\vspace{3 mm}
\begin{itemize}
\item A set of objects $Ob(\mathcal{A})$
\item For every pair of objects $x,y \in Ob(\mathcal{A})$ a $\mathbb{Z}$-graded vector space over $\mathbb{K}$ denoted by $Hom^{\bullet}_{\mathcal{A}}(x,y)$
\item For $n\ge 1$ and sequence of objects $x_0, x_1, \dots , x_n$, a graded map 
\[
m_n: Hom^{\bullet}_{\mathcal{A}}(x_{n-1}, x_n)\otimes \dots \otimes Hom^{\bullet}_{\mathcal{A}}(x_0, x_1)\to Hom^{\bullet}_{\mathcal{A}}(x_0, x_n)
\]
of degree $2-n$ such that, for every $n\ge 1$
\begin{equation}
\sum_{n=r+t+s} (-1)^{sr+t} m_{r+t+1}(Id^{\otimes^r}\otimes m_s \otimes Id^{\otimes^t})=0
\end{equation}
\item For every object $x \in Ob(\mathcal{A})$ a degree $0$ element $1_x\in Hom^{\bullet}_{\mathcal{A}}(x,x)$, called the identity at $x$, such that
\[
m_2(1_x\otimes a)=a=m_2(a\otimes 1_x)
\]
and for $n>2$, $1\le j\le n$
\[
m_n(a_1\otimes \dots \otimes a_{j-1}\otimes 1_x\otimes a_{j+1}\otimes \dots a_n)=0
\]
\end{itemize}
\end{defi}
\begin{defi}[$\mathcal{A}_∞$-functor]
Let $\mathcal{A}$ and $\mathcal{B}$ be two $\mathcal{A}_{\infty}$-categories. An $\mathcal{A}_{\infty}$-functor
\[
f: \mathcal{A}\to \mathcal{B}
\]
is the data of:
\vspace{3 mm}
\begin{itemize}
\item A map of sets $f_0:Ob(\mathcal{A})\to Ob(\mathcal{B})$
\item For $n\ge 1$ and sequence of objects $x_0, x_1, \dots , x_n$, a graded map 
\[
f_n: Hom^{\bullet}_{\mathcal{A}}(x_{n-1}, x_n)\otimes \dots \otimes Hom^{\bullet}_{\mathcal{A}}(x_0, x_1)\to Hom^{\bullet}_{\mathcal{B}}(f_0(x_0),f_0(x_n))
\]
of degree $1-n$ such that, for every $n\ge 1$
\[
\sum_{n=r+t+s} (-1)^{sr+t} f_{r+t+1}(Id^{\otimes^r}\otimes m_s \otimes Id^{\otimes^t})=\sum_{\substack{1\le r\le n\\ i_1+\dots +i_r=n}} (-1)^{\epsilon_r} m'_{r}(f_{i_1}\otimes \dots \otimes f_{i_r})
\]
where
\[
\epsilon_r=\epsilon_r(i_1,\dots ,i_r)=\sum_{2\le k\le r}\left( (1-i_k)\sum_{1\le l\le k-1} i_l \right)
\]
\item For every object $x \in Ob(\mathcal{A})$ 
\[
f_1(1_x)=1_{f_0(x)}
\]
and for $n>1$, $1\le j\le n$
\[
f_n(a_1\otimes \dots \otimes a_{j-1}\otimes 1_x\otimes a_{j+1}\otimes \dots a_n)=0
\]
\end{itemize}
\end{defi}
As for the case of morphisms of $\mathcal{A}_{\infty}$-algebras [Rem.\ref{rem1}], there is a well defined composition law for $\mathcal{A}_{\infty}$-functors defining the category of (small) $\mathcal{A}_{\infty}$-categories over a field $\mathbb{K}$, denoted by $\mathcal{A}_{\infty}Cat_{\mathbb{K}}$. 

\begin{ex}\normalfont
A differential graded category (dg-category) over a field $\mathbb{K}$ is a category enriched over the symmetric monoidal category \((Ch^{\bullet}(Vect_{\mathbb{K}}), \bigotimes, 1 )\) of cochain complexes of vector spaces over $\mathbb{K}$. Given a dg-category $\mathcal{D}$, its underlying $\mathbb{K}$-linear category, denoted by $un(\mathcal{D})$, has the same objects of $\mathcal{D}$ and vector space of morphisms between two objects $x,y$ given by
\[
Hom_{un(\mathcal{D})}(x,y)=Hom_{Ch^{\bullet}(Vect_{\mathbb{K}})}(\mathbb{K}, Hom^{\bullet}_{\mathcal{D}}(x,y))=Z^0(Hom^{\bullet}_{\mathcal{D}}(x,y))
\]
where $Z^{0}(Hom^{\bullet}_{\mathcal{D}}(x,y))$ is the set of $0$-cocycles of the complex $Hom^{\bullet}_{\mathcal{D}}(x,y)$. Given a dg-category $\mathcal{D}$, its homotopy category, denote by $H^0(\mathcal{D})$ is the $\mathbb{K}$-linear category with the same objects of $\mathcal{D}$ and vector space of morphisms between two objects $x,y$ given by
\[
Hom_{H^{0}(\mathcal{D})}(x,y)=H^0(Hom^{\bullet}_{\mathcal{D}}(x,y))
\]
In particular, composition and identity morphisms are compatible with cohomology. This construction  identifies two closed morphisms of degree $0$ if their difference is the differential of a morphism of degree $-1$. We recover the notion of dg-category considering $\mathcal{A}_{\infty}$-categories with $m_n=0$, for $n>2$, and the notion of dg-functor from the one of $\mathcal{A}_{\infty}$-functors with $f_n=0$, for $n>1$. In particular, there is faithful functor
\[
i:dgCat_{\mathbb{K}}\to \mathcal{A}_{\infty}Cat_{\mathbb{K}} 
\]  
where $dgCat_{\mathbb{K}}$ is the category of (small) differential graded categories over the field $\mathbb{K}$.
\end{ex}
\begin{ex}\normalfont
The category $Ch^{\bullet}(Vect_{\mathbb{K}})$ of cochain complexes over a field $\mathbb{K}$ has a dg-enrichment given by 
\[
Hom^{k}_{Ch^{\bullet}(Vect_{\mathbb{K}})}(A^{\bullet},B^{\bullet})=\bigoplus_{i\in \mathbb{Z}}Hom(A^i, B^{i+k})
\]
with differential 
\[
d(f_i^k)= f_{i+1}^k d_{A^i} + (-1)^{k+1} d_{B^{i+k}} f_i^k 
\] 
for $f_i^k\in Hom(A^i, B^{i+k})$. In particular, two morphisms of complexes $f^{\bullet}$ and $g^{\bullet}$ are identified in $H^{0}(Ch^{\bullet}(Vect_{\mathbb{K}}))$ if and only if they are chain homotopic. 
\end{ex}
\begin{ex}\label{ex1}\normalfont
Let $\mathcal{D}$ be a dg-category and $dgFun(\mathcal{D},Ch^{\bullet}(Vect_{\mathbb{K}}))$ be the category of dg-functors. Consider, for $k\in \mathbb{Z}$, the shift by $k$ functor on $Vect_{\mathbb{K}}$ 
\[
[k]: Ch^{\bullet}(Vect_{\mathbb{K}})\to Ch^{\bullet}(Vect_{\mathbb{K}})
\]
given on objects by 
\[
A^{\bullet}[k]^i=A^{k+i}
\] 
with differential 
\[
d[k]^i=d^{i+k}
\] 
\newpage Given a dg-functor $F\in dgFun(\mathcal{D},Ch^{\bullet}(Vect_{\mathbb{K}}))$, define its $k$-th shift by
\begin{equation}\label{sh}
F[k]=[k]\circ F
\end{equation}
This allows to define a dg-enrichment of the category $dgFun(\mathcal{D},Ch^{\bullet}(Vect_{\mathbb{K}}))$ given by
\[
Hom^{k}_{dgFun(\mathcal{D},Ch^{\bullet}(Vect_{\mathbb{K}}))}(F,G)=\{\eta: F\to G[k]\}
\]
where $\eta$ is a natural transformation of enriched functors. The differential in degree $k$ is given by the commutator
\[
d(\eta)=d \eta + (-1)^{k+1} \eta d
\]
\end{ex}
\newpage
\subsection{Simplicial nerve of an $\mathcal{A}_∞$-category and the small dg-nerve}
In this section we define a cosimplical $\mathcal{A}_∞$-category, generated by the standard simplicies and denoted by $\mathcal{A}_∞[\Delta^-]$. This construction allows to define a functor
\[
N_{\mathcal{A}_∞}: \mathcal{A}_∞Cat_{\mathbb{K}}\to SSet
\]
that we will call the simplicial nerve of an $\mathcal{A}_∞$-category. We prove that $N_{\mathcal{A}_∞}(\mathcal{A})$ is an $\infty$-category \cite{Lurie 1} for every $\mathcal{A}\in \mathcal{A}_∞Cat_{\mathbb{K}}$.
\begin{defi}
The $\mathcal{A}_∞$-category generated by the standard $n$-simplex, denoted by $\mathcal{A}_∞[\Delta^n]$, is the $\mathcal{A}_∞$-category defined by:
\vspace{3 mm}
\begin{itemize}
\item $Ob(\mathcal{A}_∞[\Delta^n])=\{0,1,\dots ,n\}$
\item For $0\le i,j\le n$
\[ 
Hom^{\bullet}_{\mathcal{A}_∞[\Delta^n]}(i, j) = \left\{
  \begin{array}{l l}
    \mathbb{K}\cdot (i,j) & i\le j\\
    \emptyset & i>j
  \end{array} \right.
\]
where deg\textnormal{((}i, j\textnormal{))}=$0$.
\item
Graded morphisms given by
\[ 
\left\{
  \begin{array}{l l}
    m_1=0 \\
    m_2((jk),(ij))=(ik), & \quad \text{for $i\le j\le k$}\\
    m_n=0, & \quad \text{for $n>2$}
  \end{array} \right.
\]
\end{itemize}
\end{defi}
Note that $\mathcal{A}_∞[\Delta^n]$ is in fact a dg-category. 
\begin{prop}
The construction $[n]\to \mathcal{A}_∞[\Delta^n]$ yields to a functor 
\[
\mathcal{A}_{\infty} [\Delta^{-}]: \Delta \to \mathcal{A}_∞ Cat_{\mathbb{K}}
\]
defining a cosimplicial $\mathcal{A}_∞$-category.
\end{prop}
\begin{proof}
Consider the standard coface and codegeneracy morphisms in $\Delta$
\begin{align*}
&\delta_j^n: [n-1] \to [n], & \quad \text{$0\le j\le n$}\\
&\sigma_j^n: [n] \to [n-1], & \quad \text{$0\le j\le n-1$} \,
\end{align*}
given  by
\[ 
\delta_j^n(k) = \left\{
  \begin{array}{l l}
     k, & 0\le k\le j-1\\
    k+1, & j\le k\le n-1
  \end{array} \right.
\]
\[ 
\sigma_j^n(k) = \left\{
  \begin{array}{l l}
     k, & 0\le k\le j\\
    k-1, & j+1\le k\le n
  \end{array} \right.
\]
Define the induced coface $\mathcal{A}_∞$-functors
\[
(\delta_j^n)_{\star}: \mathcal{A}_∞[\Delta^{n-1}] \to \mathcal{A}_∞[\Delta^n]
\]
by
\[ 
\left\{
  \begin{array}{l l}
    (\delta_j^n)_{\star,0}(k)=\delta_j^n(k)\\
    (\delta_j^n)_{\star,1}(i,k)=(\delta_j^n(i), \delta_j^n(k))\\
    (\delta_j^n)_{\star,p}=0, & \quad \text{$p>2$}
  \end{array} \right.
\]
and similarly the codegeneracy $\mathcal{A}_∞$-functors 
\[
(\sigma_j^n)_{\star}: \mathcal{A}_∞[\Delta^{n}] \to \mathcal{A}_∞[\Delta^{n-1}]
\]
by
\[ 
\left\{
  \begin{array}{l l}
    (\sigma_j^n)_{\star,0}(k)=\sigma_j^n(k)\\
    (\sigma_j^n)_{\star,1}(i,k)=(\sigma_j^n(i), \sigma_j^n(k))\\
    (\sigma_j^n)_{\star,p}=0, & \text{$p>2$}
  \end{array} \right.
\]
It is straightforward to check that this construction is functorial. 
\end{proof}
\vspace{3 mm}
\begin{defi}[Simplicial Nerve of an $\mathcal{A}_∞$-category]
Let $\mathcal{A}$ be an $\mathcal{A}_∞$-category. The simplicial nerve $N_{\mathcal{A}_∞}(\mathcal{A})$ is the simplicial set whose $n$-simplicies are given by
\[
N_{\mathcal{A}_∞}(\mathcal{A})_n=Hom_{\mathcal{A}_∞ Cat_{\mathbb{K}}}(\mathcal{A}_{\infty} [\Delta^n],\mathcal{A})
\]
and simplicial structure induced by applying the functor $Hom_{\mathcal{A}_∞ Cat_{\mathbb{K}}}(-,\mathcal{A})$ to the cosimplicial $\mathcal{A}_∞$-category $\mathcal{A}_{\infty} [\Delta^{-}]$.
 \end{defi}
\begin{rem}\normalfont
We give a more explicit description of what an $n$-simplex of the simplicial nerve is. Let $f\in N_{\mathcal{A}_∞}(\mathcal{A})_n$, then by definition
\[
f:\mathcal{A}_{\infty} [\Delta^n]\to \mathcal{A}
\]
is an $\mathcal{A}_{\infty}$-functor and hence it is given by a set theoretic function
\[
f_0:Ob(\mathcal{A}_{\infty} [\Delta^n])\to Ob(\mathcal{A})
\]
corresponding to the choice of $n+1$ object of $\mathcal{A}$
\[
x_i=f_0(i)\in Ob(\mathcal{A})
\]
and, for every $1\le k\le n$, a graded morphism of degree $1-k$
\[
f_k: Hom^{\bullet}_{\mathcal{A}_{\infty} [\Delta^n]}(x_{i_{k-1}}, x_{i_k})\otimes \dots \otimes Hom^{\bullet}_{\mathcal{A}_{\infty} [\Delta^n]}(x_{i_0}, x_{i_1})\to Hom^{\bullet}_{\mathcal{A}}(f_0(x_{i_0}),f_0(x_{i_k}))
\]
corresponding to the choice, for every $1\le k\le n$ and every string $0\le i_0<i_1<\dots <i_k\le n$, of an element
\[
f_{i_0\dots i_k}=f_k((i_{k-1}i_k)\otimes \dots \otimes (i_0i_1))\in Hom^{1-k}_{\mathcal{A}}(x_{i_0},x_{i_k})
\]
Moreover, if at least one the $i_k$'s is repeated, by strict unitality we have
\[ 
\left\{
  \begin{array}{l l}
    f_{i_0,i_0}=Id_{x_{i_0}}\\
   f_{i_0,\dots ,i_p,i_p,\dots ,i_k}=0, & \quad \text{for $2\le k\le n$}
  \end{array} \right.
\]
The system of equations, for $n\ge 1$
\[
\sum_{n=r+t+s} (-1)^{sr+t} f_{r+t+1}(Id^{\otimes^r}\otimes m_s \otimes Id^{\otimes^t})=\sum_{\substack{1\le r\le n\\ p_1+\dots +p_r=n}} (-1)^{\epsilon_r} m'_{r}(f_{p_1}\otimes \dots \otimes f_{p_r})
\]
can be written in terms of those data, for every $1\le k\le n$ and every string $0\le i_0<i_1<\dots <i_k\le n$, as
\begin{align*}
&m_1(f_{i_0\dots i_k})=\sum_{0<j<n}(-1)^{j-1} f_{i_0\dots \hat{i_j}\dots i_k} + \sum_{0<j<n}(-1)^{1+k(j-1)}m_2(f_{i_j\dots i_k},f_{i_0\dots i_j})+\\
&+\sum_{\substack{1\le r\le n\\ 0< j_1<\dots < j_{r-1}< n}} (-1)^{1+\epsilon_r} m_{r}(f_{i_{j_{r-1}}\dots i_k},\dots ,f_{i_0\dots i_{j_1}}) \,
\end{align*}
where
\[
\epsilon_r=\epsilon_r(j_1,\dots ,j_r)=\sum_{2\le k\le r} (1-j_k+j_{k-1}) j_{k-1} 
\]
with $j_n=n$. 
\end{rem}
\begin{rem}\normalfont
We compute the simplicial structure of $N_{\mathcal{A}_∞}(\mathcal{A})$. The $j$-th face map
\[
d_j^n:N_{\mathcal{A}_∞}(\mathcal{A})_n\to N_{\mathcal{A}_∞}(\mathcal{A})_{n-1}
\]
is computed by 
\[
d_j^n(f)=f\circ (\delta_j^n)_{\star}
\]
Recall that given $f$ and $g$ functors of $\mathcal{A}_∞$-categories, their composition is given by
\[
(f\circ g)_k=\sum_{r=1}^k \sum_{i_1+\dots +i_r=k}(-1)^{\epsilon_r(i_1,\dots ,i_r)} f_{r}(g_{i_1}\otimes \dots \otimes g_{i_r})
\]
In particular, if $g$ is a functor of dg-categories then the composition is given by
\[
(f\circ g)_k=f_{k}(g_{1}\otimes \dots \otimes g_{1})
\]
hence
\[
(f\circ (\delta_j^n)_{\star})_k=f_{k}((\delta_j^n)_{\star,1}\otimes \dots \otimes (\delta_j^n)_{\star,1})
\]
In particular, given $1\le k\le n$ and a string $0\le i_0<i_1<\dots <i_k\le n-1$, we find that
\[ 
d_j^n(f)_{i_0\dots i_k} = \left\{
  \begin{array}{l l}
     f_{i_0\dots i_{p-1}(i_p+1)\dots (i_k+1)}, & j\le i_p, 0\le p\le k\\
    f_{i_0\dots i_k}, & j> i_k
  \end{array} \right.
\]
Similarly, the $j$-th degeneracy map
\[
s_j^n:N_{\mathcal{A}_∞}(\mathcal{A})_{n-1}\to N_{\mathcal{A}_∞}(\mathcal{A})_n
\]
is given by
\[ 
s_j^n(f)_{i_0 i_1} = \left\{
  \begin{array}{l l}
     f_{(i_0-1)(i_1-1)}, & j\le i_0-1\\
    f_{i_0(i_1-1)}, & i_0<j< i_1-1 \\
    Id_{x_{i_0}}, & i_0=j, i_1=j+1 \\
    f_{i_0i_1}, & j\ge i_1
  \end{array} \right.
\]
and for $k\ge 2$
\[ 
s_j^n(f)_{i_0\dots i_k} = \left\{
  \begin{array}{l l}
     f_{(i_0-1)\dots (i_k-1)}, & j\le i_0-1\\
    f_{i_0\dots i_p(i_{p+1}-1)\dots (i_k-1)}, & i_p<j< i_{p+1}-1, 0<p<k \\
    0, & i_p=j, i_{p+1}=j+1 \\
    f_{i_0\dots i_k}, & j\ge i_k
  \end{array} \right.
\]
\end{rem}
\begin{rem}\normalfont
An $\mathcal{A}_{\infty}$-functor
\[
g:\mathcal{A}\to \mathcal{B}
\]
induces a map of simplicial sets
\[
(g)_{\star}: N_{\mathcal{A}_∞}(\mathcal{A})\to N_{\mathcal{A}_∞}(\mathcal{B})
\]
by 
\[
(g)_{\star}(f)=f\circ g
\]
for $f\in Hom_{\mathcal{A}_∞ Cat_{\mathbb{K}}}(\mathcal{A}_{\infty} [\Delta^n],\mathcal{A})$. More explicitly, if $1\le k\le n$ 	and $0\le i_0<i_1<\dots <i_k\le n$ we have 
\[
((g)_{\star}(f))_{i_0\dots i_k}=\sum_{r=1}^k \sum_{j_1+\dots +j_r=k}(-1)^{\epsilon_r(j_1,\dots ,j_r)} g_{r}(f_{i_{j_r+\dots +j_2}\dots i_k}, \dots, f_{i_{0}\dots i_{j_r}})
\]
In particular we have the following proposition
\end{rem}
\begin{prop} 
The simplicial nerve of an $\mathcal{A}_{\infty}$-category defines a functor
\[
N_{\mathcal{A}_∞}: \mathcal{A}_∞Cat_{\mathbb{K}}\to SSet
\]
\qed
\end{prop}
\vspace{3 mm}
\begin{defi}[Small dg-nerve]
Let $\mathcal{D}$ be a dg-category. The small dg-nerve of $\mathcal{D}$ is the $\mathcal{A}_{\infty}$-nerve of $i(\mathcal{D})$
\[
N_{dg}^{sm}(\mathcal{D})=N_{\mathcal{A}_∞}(i(\mathcal{D}))
\]
\end{defi}
The small dg-nerve makes functorial the construction of J.Lurie \citep[Chap. 3, \S 1.3.1]{Lurie 2}.
\[
N_{dg}^{sm}: dgCat_{\mathbb{K}}\to SSet
\]
We remark the fact that this construction has the right signs in the expression of the differential
\begin{equation}\label{dgeq}
d(f_{i_0\dots i_k})=\sum_{0<j<k}(-1)^{j-1} f_{i_0\dots \hat{i_j}\dots i_k} + \sum_{0<j<k}(-1)^{1+k(j-1)} f_{i_j\dots i_k}\cdot f_{i_0\dots i_j}
\end{equation}
in the sense that $d^2(f_{i_0\dots i_k})=0$. We recall now the notion of  $\infty$-category and prove that the simplicial nerve of an $\mathcal{A}_∞$-category is an $∞$-category. 
 \begin{defi}[$\infty$-category, \cite{Lurie 1}]
An $\infty$-category is a simplicial set $X$ such that, for any 0$<$ i $<$n and any map of simplicial sets $f:\Lambda^{n}_{i} \to X$, there exists an extension to the full n-simplex  $g:\Delta^{n} \to X$
\[
  \begin{tikzpicture}
    \def\x{1.5}
    \def\y{-1.2}
    \node (A1_1) at (1*\x, 1*\y) {$\Lambda_{i}^{n}$};
    \node (A2_1) at (2*\x, 1*\y) {$X$};
    \node (A1_2) at (1*\x, 2*\y) {$\Delta^{n}$};
   \path (A1_1) edge [right hook->] node [auto] {$\scriptstyle{i}$} (A1_2);
    \path (A1_1) edge [->] node [auto,swap] {$\scriptstyle{f}$} (A2_1);
    \path (A1_2) edge [->, dashed] node [auto,swap] {$\scriptstyle{g}$} (A2_1);
      \end{tikzpicture}
  \]
where $\Lambda^{n}_{i}$ is the $i$-th inner horn. This property is called left lifting property for inner horns.
\end{defi}
Given an  $\infty$-category $X$, the $0$-simplicies of $X$ should be thought as the objects of the $\infty$-category and the $n$-simplicies as $n$-morphisms. In particular, the inner horn filling property induces a weak composition law associative up to higher degree simplicies \citep[see][Chap. 1]{Lurie 1}. 
\begin{ex}\normalfont
Let $\mathcal{C}$ be a small category, its nerve $N_{Cat}(\mathcal{C})$ is the simplicial set given by
\[
N_{Cat}(\mathcal{C})_{0}=Ob(\mathcal{C})
\]
and for $n>0$
\[
N_{Cat}(\mathcal{C})_{n}=\{f_{i}: x_{i} \to x_{i+1}, x_i\in Ob({\mathcal{C}}), f_i\in Hom_{\mathcal{C}}(x_i,x_{i+1}), i=0,..,n-1\}
\]
one can easily prove that this is an $\infty$-category and that the nerve defines a functor
\[
N_{Cat}: Cat \to SSet
\]
whose essential image is characterized by those simplicial sets for whch
\[
X_n\simeq X_1 \underset{s,X_0,t} {\times} X_1\dots X_1 \underset{s,X_0,t} {\times} X_1
\]
where the right hand side is the $n$-th folded fibred product in $Sets$ taken over source and target maps of $X_{\bullet}$.
\end{ex}  
\begin{prop}
Given an $\mathcal{A}_{\infty}$-category $\mathcal{A}$, its simplicial nerve $N_{\mathcal{A}_{\infty}}(\mathcal{A})$ is an $\infty$-category
\end{prop}
\begin{proof}
Fix $n>0$ and $0<p<n$ and a morphism of simplicial sets $\gamma: \Lambda_p^n \to N_{\mathcal{A}_{\infty}}(\mathcal{A})$. Such morphism can be identified with an $n$-simplex in $N_{\mathcal{A}_{\infty}}(\mathcal{A})$
\[
(\{x_{i_0}\}_{0\le i_0\le n}, \{f_{i_0i_1}\}_{0\le i_0\le i_1\le n}, \dots, \{f_{0\dots n}\})
\]
where the components $f_{0\dots n}$ and $f_{0\dots \hat{p} \dots n}$ are not given. The data $f_{0\dots n}=0$ and
\begin{align*}
&f_{0\dots \hat{p} \dots n}= \sum\limits_{0<j<n, j\ne p} (-1)^{j-1+p} f_{0\dots \hat{j}\dots n} + \sum\limits_{0<j<n} (-1)^{1+n(j-1)+p} f_{j\dots n}\circ f_{0\dots j}+\\
&+\sum_{\substack{1\le r\le n\\ 0< j_1<\dots < j_{r-1}< n}} (-1)^{1+\epsilon_r(j_1,\dots ,j_{r-1})+p} m_{r}(f_{i_{j_{r-1}}\dots i_k},\dots ,f_{i_0\dots i_{j_1}}) \,
\end{align*}
provide a filling for $\gamma$.
\end{proof}
\newpage
There is an equivalent way to define $\infty$-categories as categories enriched over the symmetric monoidal category of simplicial sets.
 \begin{defi}[Simplicial category]
A simplicial category is a category enriched over the symmetric monoidal category \((SSet, \times, pt )\) of simplicial sets, where the monoidal structure is the pointwise cartesian product of simplicies. 
\end{defi}
As for the case of dg-categories, when can recover the set of morphisms in the underlying category of a simplicial category $\mathcal{C}$ by
\[
Hom_{\mathcal{C}}(x,y)=Hom_{SSet}(pt, Map_{\mathcal{C}}(x,y))=Map_{\mathcal{C}}(x,y)_{0}
\]
Simplicial categories can be related to simplicial sets. Namely, J.Lurie \citep[Chap. 2, \S 2.2]{Lurie 1} proves that there is a pair of adjoint functors
\begin{equation}\label{adj}
\begin{tikzpicture}
\matrix (m) [matrix of math nodes, row sep=2em, column sep=3pc,
  text width=2.2pc, text height=1pc, text depth=.5pc] { 
   SCat & SSet \\
  }; 
\path[<-] 
(m-1-1.15) edge node[above] {\tiny $\mathcal{C}[-]$}  (m-1-2.165)
(m-1-2.-165) edge node[below] {\tiny $ N_{SCat}$}  (m-1-1.-15);
\end{tikzpicture}
\end{equation}
More precisely, there exist model structures on both sides of (\ref{adj}) with respect to which this adjunction extends to a Quillen adjunction. Fibrant simplicial sets for this model structure, called the Joyal model structure, are $\infty$-categories. Fibrant simplicial categories are simplicial categories for which all the mapping spaces are Kan complexes. We refer to \citep[Chap. 1]{Lurie 1} for a more detailed discussion.
\begin{rem}[Homotopy category of an $\infty$-category] \normalfont
Given an $\infty$-category $X$, its homotopy category $h(X)$ is the category with objects the 0-simplicies of $X$ and set of morphisms given by
\[
Hom_{h(X)}(x,y)=\pi_{0}(Map_{\mathcal{C}[X]}(x,y))
\]
We refer to \citep[Chap. 1]{Lurie 1} for more details about the construction of the homotopy category. As an example, it is easy to check that if $\mathcal{D}$ is a dg-category, then there is an isomorphism $h(N_{dg}^{sm}(\mathcal{D}))\simeq H^0(\mathcal{D})$.
\end{rem}
\newpage
\section{The big dg-nerve of a dg-category}
We saw in the previous section how to associate to a dg-category $\mathcal{D}$ an $\infty$-category, the small dg-nerve $N_{dg}^{sm}(\mathcal{D})$. We present in this section a different $\infty$-category associated to a dg-category, called the big dg-nerve. J.Lurie in \cite{Lurie 2} proves that those are equivalent $\infty$-category. We provied a more detailed description of how to construct such equivalence.
\subsection{The Dold-Kan correspondence}
We recall a classical tool in homological algebra that allows to compare complexes in an abelian category $\mathcal{A}$ with simplicial objects in $\mathcal{A}$. 
\begin{thm}
Let $\mathcal{A}$ be an abelian category, then there is an equivalence of categories, called Dold-Kan Correspondence
\[
DK: Ch_{\bullet}^{\ge0}(\mathcal{A}) \to S(\mathcal{A})
\]
between the category of positively graded chain complexes in $\mathcal{A}$ and the category of simplicial objects in $\mathcal{A}$. Its inverse equivalence
\[
N: S(\mathcal{A}) \to Ch_{\bullet}^{\ge0}(\mathcal{A})
\]
is called the normalized chain complex associated to a simplical object of $\mathcal{A}$. Moreover, under this equivalence, homology and homotopy groups are preserved. 
\begin{align*}
&\pi_{i}(DK(A_{\bullet})) \simeq H_{i}(A_{\bullet})\\
&H_{i}(N(X_{\bullet})) \simeq \pi_{i}(X_{\bullet}) \,
\end{align*}
for $i\ge 0$.
\qed
\end{thm}
For future use, we recall the construction of the functors $N$ and $DK$. First, if $X_{\bullet}$ is a simplicial object in $\mathcal{A}$, then its normalized chain complex is given by
\[
N(X_{\bullet})_{n}=\bigcap_{i=0,\dots,n-1} Ker(d_{i}: X_{n} \to X_{n-1})
\]
with differential $d_{n}=(-1)^n d_{n}^n$. It is possible to prove that there is a natural isomorphism of functors
\[
N(X_{\bullet}) \simeq \frac{C(X_{\bullet})}{D(X_{\bullet})}
\]
where $C(X_{\bullet})_{n}=X_{n}$ and $D(X_{\bullet})_{n}=\sum\nolimits_{i=0}^n \sigma_{i}(X_{n-1})$ with differential given by $d_{n}=\sum\nolimits_{i=0}^n (-1)^{i} d_{i}^n$. Second, if $A_{\bullet}$ is a chain complex in $\mathcal{A}$, then
\[
DK(A_{\bullet})_{n}=\bigoplus_{p \le n} \bigoplus_{\eta : [n] \twoheadrightarrow [p]} A_{p}[\eta]
\]
where $A_{p}[\eta]=A_{p}$. If $\theta: [m]\to [n]$ is a morphism in $\Delta$, the induced map 
\[
\theta^{\ast}:DK(A_{\bullet})_{n} \to DK(A_{\bullet})_{m}
\] 
is given, on the component associated to $\eta : [n] \twoheadrightarrow [p]$, by considering an epi-mono factorization of the morphism $\eta \circ \theta: [m] \to [p]$ 
\[
\eta \circ \theta = \rho \circ \epsilon
\] 
where $\epsilon : [m] \twoheadrightarrow [q]$ and $\rho : [q] \hookrightarrow [p]$ and by taking the composition 
\[
 A_{p}[\eta] \to  A_{q}[\epsilon] \hookrightarrow DK(A_{\bullet})_{m}
\]
where the first map is 
\[
\begin{cases} (-1)^p \partial_{p}, &\mbox{if } q=p-1, \rho=d_{p}^p \\
Id_{A_p}, & \mbox{if } \rho=Id_q\\ 
0, &\mbox{otherwise } . 
\end{cases}  
\]
\begin{ex}\normalfont
Low degree terms of $DK(A_{\bullet})$ are given by
\begin{align*}
&DK(A_{\bullet})_0=A_0\\
&DK(A_{\bullet})_1=A_0 \oplus A_1 \\
&DK(A_{\bullet})_2=A_0 \oplus A_1^1 \oplus A_1^2 \oplus A_2 \, .
\end{align*}
with face maps given by
\begin{align*}
&d_0^1(a_0,a_1)=a_0\\
& d_1^1(a_0,a_1)=a_0-d_1(a_1) \,
\end{align*}
and
\begin{align*}
&d_0^2(a_0,a_1^1,a_2^1,a_2)=(a_0,a_1^2)\\
&d_1^2(a_0,a_1^1,a_2^1,a_2)=(a_0,a_1^1+a_1^2)\\
& d_2^2(a_0,a_1^1,a_2^1,a_2)=(a_0-d_1(a_1^2), a_1^1 + d_2(a_2)) \,
\end{align*}
We remark the fact that, if $\bar{a}\in DK(A_{\bullet})_{n}$ and 
\[
\pi_n:DK(A_{\bullet})_{n}\to A_{n}
\]
is the projection onto the highest degree component. Then, by construction, the following equation holds:
\begin{equation}\label{DK}
d_n(\pi_n(\bar{a}))=\sum_{j=0}^n (-1)^j \pi_{n-1}(d_j^n(\bar{a}))
\end{equation} 
\end{ex}
\begin{rem}\normalfont \label{mon}
 Let $\mathbb{K}$ be a commutative ring and consider the abelian category $\mathcal{A}=Mod_{\mathbb{K}}$ of $\mathbb{K}$-modules. Dold-Kan correspondence provides an equivalence of categories between $Ch_{\bullet}^{\ge0}(Mod_{\mathbb{K}})$ and $S(Mod_{\mathbb{K}})$. We now recall the relations of such equivalence with the monoidal structure. Namely, both categories are endowed with a symmetric monoidal structure with tensor product functor given, on $Ch_{\bullet}^{\ge0}(Mod_{\mathbb{K}})$ by
\[
(A_{\bullet} \bigotimes B_{\bullet})_{n}= \bigoplus_{p+q=n} A_{p} \otimes B_{q}
\]
with differential given by the graded commutative Libeniz 
\[
d(a \otimes b)= d(a) \otimes b + (-1)^{deg(a)} a \otimes d(b)
\]
and on the category of simplicial $\mathbb{K}$-modules by
\[
(X_{\bullet} \times Y_{\bullet})_{n}=  X_{n} \otimes Y_{n}
\]
Moreover, there are natural morphisms
\[
\nabla_{(X_{\bullet},Y_{\bullet})}: N(X_{\bullet}) \bigotimes N(Y_{\bullet}) \to N(X_{\bullet} \times Y_{\bullet}) 
\]
\[
\Delta_{(X_{\bullet},Y_{\bullet})}:  N(X_{\bullet} \times Y_{\bullet}) \to N(X_{\bullet}) \bigotimes N(Y_{\bullet}) 
\]
where $\nabla_{(X_{\bullet},Y_{\bullet})}$ is called the Eilenberg-Zilber map given by 
\[
\nabla_{(X_{\bullet},Y_{\bullet})}(x_{p} \otimes y_{q})=\sum_{(\mu,\nu)} sign(\mu,\nu) s_{\nu}(x_{p}) \otimes s_{\mu}(y_{q})
\]
where the sum is taken over all the $(p,q)$-shuffles and if $(\mu,\nu)=(\mu_1,\dots,\mu_p,\nu_1,\dots,\nu_q)$ then $s_{\mu}=s_{\mu_p} \dots s_{\mu_1}$ and similarly $s_{\nu}=s_{\nu_q} \dots s_{\nu_1}$. $\Delta_{(X_{\bullet},Y_{\bullet})}$ is called the Alexander-Whitney map and is given by
\begin{equation}\label{AW}
\Delta_{(X_{\bullet},Y_{\bullet})}(x_{n} \otimes y_{n})=\bigoplus_{p+q=n} (-1)^{np+1} \tilde{d}^p(x_{n}) \otimes d^{q}_{0}(y_{n})
\end{equation}
where $\tilde{d}^p$ is the map induced by the morphism in $\Delta$, $\tilde{\delta}^p: [p] \to [n]$ given by $\tilde{\delta}^p(i)=i$, and $ d^{q}_{0}$ is induced by $ \delta^{q}_{0}: [q] \to [n]$ given by $\delta^{q}_{0}(i)=p+i$. 

The Eilenberg-Zilber map is a symmetric lax monoidal transformation. The Alexander-Whitney map is just an oplax monoidal transformation. It easy to check that the composition $\Delta_{(X_{\bullet},Y_{\bullet})}\circ \nabla_{(X_{\bullet},Y_{\bullet})}$ is the identity map and that $\nabla_{(X_{\bullet},Y_{\bullet})} \circ \Delta_{(X_{\bullet},Y_{\bullet})}$ is homotopic to the identity  \cite{MacLane}. The natural isomorphisms of functors $\epsilon: DK \circ N \to Id$ and $\eta: N \circ DK \to Id$ induces natural transformations
\[
\tilde{\nabla}_{(A_{\bullet},B_{\bullet})}: DK(A_{\bullet}) \times DK(B_{\bullet}) \to DK(A_{\bullet} \otimes B_{\bullet}) 
\]
\[
\tilde{\Delta}_{(A_{\bullet},B_{\bullet})}:  DK(A_{\bullet} \otimes B_{\bullet}) \to DK(A_{\bullet}) \times DK(B_{\bullet}) 
\]
In particular, $\tilde{\Delta}$ makes the functor $DK$ an oplax monoidal functor \cite{Schw}. 
\end{rem}
Consider now the functor 
\[
op: Ch^{\bullet}(Vect_{\mathbb{K}}) \to Ch_{\bullet}(Vect_{\mathbb{K}}) 
\]
sending a cochain complex of vector spaces in the chain complex $A^{op}_{n}=A^{-n}$ and the truncation functor
\[
\tau_{\ge0}: Ch_{\bullet}(Vect_{\mathbb{K}}) \to Ch_{\bullet}^{\ge0}(Vect_{\mathbb{K}})
\] 
defined by 
\[
\tau_{\ge0}(A)_n = \begin{cases} 0 &\mbox{if } n < 0 \\
Ker(d_0) & \mbox{if } n=0\\ A_n &\mbox{if } n > 0. \end{cases}  
\]
Both functors are clearly compatible with the monoidal structure. In particular, given a morphism of chain complexes
\[
A^{\bullet}\otimes B^{\bullet}\to C^{\bullet}
\]
the compatibility of those functors and of the functor $DK$ with the monoidal structure allows to define a unique map of simplicial sets
\[
DK(\tau_{\ge0}((A^{\bullet})^{op}))\otimes DK(\tau_{\ge0}((B^{\bullet})^{op}))\to DK(\tau_{\ge0}((C^{\bullet})^{op}))
\]
\begin{defi}[\cite{Lurie 2}]
Let $\mathcal{D}$ be a dg-category. $\mathcal{D}_{\Delta}$ is the simplicial category with the same objects as $\mathcal{D}$ and with simplicial set of morphisms given by
\[
Map_{\mathcal{D}_{\Delta}}(x,y):= DK(\tau_{\ge0}(Hom^{\bullet}_{\mathcal{D}}(x,y)^{op}))
\]
The composition law and identity 
\[
\circ : Map_{\mathcal{D}_{\Delta}}(x,y) \otimes Map_{\mathcal{D}_{\Delta}}(y,z) \to Map_{\mathcal{D}_{\Delta}}(x,z)
\]
\[
Id: \ast \to Map_{\mathcal{D}_{\Delta}}(x,x)
\]
are induced by the composition law and identity in $\mathcal{D}$ as pointed in \normalfont [Rem.\ref{mon}] 
\end{defi}
\newpage
\subsection{Cubical interpretation of the simplicial nerve of a simplicial category} 
Recall that, according to \citep[Chap. 2, \S 2.2]{Lurie 1}, there is a pair of adjoint functors
\[
\begin{tikzpicture}
\matrix (m) [matrix of math nodes, row sep=2em, column sep=3pc,
  text width=2.2pc, text height=1pc, text depth=.5pc] { 
   SCat & SSet \\
  }; 
\path[<-] 
(m-1-1.15) edge node[above] {\tiny $\mathcal{C}[-]$}  (m-1-2.165)
(m-1-2.-165) edge node[below] {\tiny $ N_{SCat}$}  (m-1-1.-15);
\end{tikzpicture}
    \]
In this section we provide a geometric description via cubes of the mapping spaces of the simplicial category  $\mathcal{C}[\Delta^n]$. This description will be used in the next section to define an equivalence of $\infty$-categories between $N_{SCat}(\mathcal{D}_{\Delta})$ and $N^{sm}_{dg}(\mathcal{D})$, where $\mathcal{D}$ is any dg-category. 

The simplicial category $\mathcal{C}[\Delta^n]$ is a fattened version of the ordinal category on $n$ elements $[n]$ and has objects 
\[
Ob(C[\Delta^n])=\{0,1,\dots,n\}
\] 
and mapping spaces
\[
Map_{C^{\Delta^n}}(i,j)= \begin{cases} \emptyset &\mbox{if } i>j \\
N_{Cat}(P_{i,j}) & \mbox{if } i \le j. \end{cases}  
\]
where $P_{i,j}$ is the category associated to the partially ordered collection of subsets of the ordered set $\{0,1,\dots,n\}$ having minimum $i$ and maximum $j$. The composition is induced by taking union of those subsets
\[
\cup: P_{i,j} \times P_{j,k} \to P_{i,k}
\] 
and identity given by $Id_{P_{i,i}}=\{i\}$. 
\begin{ex}\normalfont
We describe the simplicial category $C[\Delta^n]$ for low values of $n$. $C[\Delta^1]$ has as set of objects $\{0,1\}$ and mapping spaces given by
\begin{align*}
&Map_{C[\Delta^1]}(0,0) \simeq Map_{C[\Delta^1]}(1,1) \simeq \{Id\} \simeq \ast \\
&Map_{C[\Delta^1]}(0,1) \simeq \{\{0,1\}\} \simeq \ast \, 
\end{align*}
The $2$-dimensional simplicial category $C[\Delta^2]$ has set of objects $\{0,1,2\}$ and mapping spaces 
\begin{align*}
&Map_{C[\Delta^2]}(0,0) \simeq Map_{C[\Delta^2]}(1,1) \simeq Map_{C[\Delta^2]}(2,2) \simeq \ast \\
&Map_{C[\Delta^2]}(0,1) \simeq \{\{0,1\}\} \simeq \ast \\
&Map_{C[\Delta^2]}(1,2) \simeq \{\{1,2\}\} \simeq \ast \\
&Map_{C[\Delta^2]}(0,2) \simeq \Delta^1 \, 
\end{align*}
where the non trivial $1$-simplex correspond to the morphism in $P_{0,2}$, $\{0,2\} \hookrightarrow \{0,1,2\}$. Moreover, the composition law gives the identity
\[
\{1,2\} \circ \{0,1\} = \{0,1,2\}
\]
hence there is a $1$-simplex in $Map_{C[\Delta^2]}(0,2)$ connecting $\{0,2\}$ to $\{1,2\} \circ \{0,1\}$. In the 3 dimensional case, $C[\Delta^3]$ has objects $\{0,1,2,3\}$, mapping spaces
\begin{align*}
&Map_{C[\Delta^3]}(0,0) \simeq Map_{C[\Delta^3]}(1,1) \simeq Map_{C[\Delta^3]}(2,2) \simeq Map_{C[\Delta^3]}(3,3) \simeq \ast \\
&Map_{C[\Delta^3]}(0,1) \simeq \{\{0,1\}\} \simeq \ast \\
&Map_{C[\Delta^3]}(1,2) \simeq \{\{1,2\}\} \simeq \ast \\
&Map_{C[\Delta^3]}(2,3) \simeq \{\{2,3\}\} \simeq \ast \,
\end{align*} 
non trivial $1$-dimensional mapping spaces
\begin{align*}
&Map_{C[\Delta^3]}(0,2) \simeq \Delta^1 \\
& Map_{C[\Delta^3]}(1,3) \simeq \Delta^1\\
& Map_{C[\Delta^3]}(0,2) \simeq \Delta^1 \\
&Map_{C[\Delta^3]}(1,3) \simeq \Delta^1 \, 
\end{align*} 
corresponding to the inclusions $\{0,2\} \hookrightarrow \{0,1,2\}$ and $\{1,3\} \hookrightarrow \{0,1,3\}$, and $2$-dimensional mapping spaces
\[
Map_{C[\Delta^3]}(0,3) \simeq \Delta^2 \coprod \Delta^2
\]
where the two $2$-dimensional simplicies can be pictured as
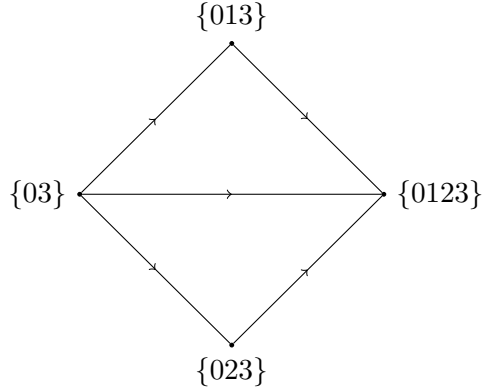
\[
\begin{tikzpicture}
[scale=1, vertices/.style={draw, fill=black, circle, inner sep=0.5pt}]
\node[vertices, label=above:{$\{013\}$}] (a) at (0,4) {};
\node[vertices, label=right:{$\{0123\}$}] (b) at (2,2) {};
\node[vertices, label=left:{$\{03\}$}] (c) at (-2,2) {};
\node[vertices, label=below:{$\{023\}$}] (d) at (0,0) {};
\begin{scope}[decoration={markings,mark=at position 0.5 with {\arrow{>}}}] 
\draw[postaction={decorate}] (c) -- node[auto,swap] {$$} (a);
\draw[postaction={decorate}] (a) -- node[auto] {$$} (b);
\draw[postaction={decorate}] (c) -- node[auto,swap] {$$} (b);
\draw[postaction={decorate}] (c) -- node[auto] {$$} (d);
\draw[postaction={decorate}] (d) -- node[auto,swap] {$$} (b);
\end{scope}
\end{tikzpicture}
\]
those simplicies satisfy the relations given by the composition law on the $0$-simplicies
\[ 
\left\{
  \begin{array}{l l}
   \{0,2\} \hookrightarrow \{0,1,2\} = \{1,2\} \circ \{0,1\} \\
   \{1,3\} \hookrightarrow \{1,2,3\} = \{2,3\} \circ \{1,2\} \\
   \{0,3\} \hookrightarrow \{0,2,3\} = \{2,3\} \circ \{0,2\} \\
   \{0,3\} \hookrightarrow \{0,1,3\} = \{1,3\} \circ \{0,1\} \\
   \{0,2\} \hookrightarrow \{0,1,2\} = \{1,2\} \circ \{0,1\} \\
   \{0,3\} \hookrightarrow \{0,1,3\} \hookrightarrow \{0,1,2,3\} = \{2,3\} \circ \{0,1,2\} \\
   \{0,3\} \hookrightarrow \{0,2,3\} \hookrightarrow \{0,1,2,3\} = \{2,3\} \circ \{0,1,2\} \\
   \{0,3\} \hookrightarrow \{0,1,3\} \hookrightarrow \{0,1,2,3\} = \{1,2,3\} \circ \{0,1\} \\
   \{0,3\} \hookrightarrow \{0,2,3\} \hookrightarrow \{0,1,2,3\} = \{1,2,3\} \circ \{0,1\}
  \end{array} \right.
\]
and on the $1$-simplicies 
\[ 
\left\{
  \begin{array}{l l}
   Id_{\{2,3\}} \circ i_{02}^{012}=i_{023}^{0123} \\
   i_{13}^{123} \circ Id_{\{0,1\}}=i_{013}^{0123}
   \end{array} \right.
\]
\end{ex}
\vspace{7 mm}
Fix $n>0$ and $0\le i\le j\le n$. We give now a geometric description of the mapping space $Map_{\mathcal{C}[\Delta^n]}(i,j)$ in terms of cubes. First notice that
\[
Map_{\mathcal{C}[\Delta^n]}(i,j)\simeq Map_{\mathcal{C}[\Delta^{j-i}]}(0,j-i)
\]
Hence, it is enough to describe $Map_{\mathcal{C}[\Delta^m]}(0,m)$, for every $m>0$. We saw that
\begin{align*}
&Map_{C[\Delta^0]}(0,0) \simeq \ast \\
&Map_{C[\Delta^1]}(0,1) \simeq \ast \, 
\end{align*} 
Let $m>1$ and let 
\[
\Sigma_{m-1}=Aut_{Set}(1,\dots, m-1)
\] 
be the symmetric group on $m-1$ letters. Given an element $\sigma \in \Sigma_{m-1}$, we associate to it an $(m-1)$-simplex in $Map_{C[\Delta^m]}(0,m)$ considering the chain of inclusions
\[
\{0,m\} \hookrightarrow \{0,\sigma(1),m\} \hookrightarrow \{0,\sigma(1), \sigma(2),m\} \hookrightarrow \dots \hookrightarrow \{0,1, \dots,m\}
\]
call this $(m-1)$-simplex $\Delta_{\sigma}^{m-1}$.
\begin{prop}
Let $m>1$, then the simplicial set $Map_{\mathcal{C}[\Delta^m]}(0,m)$ is identified with
\[
\coprod_{\sigma \in \Sigma_{m-1}} \Delta_{\sigma}^{m-1}
\]
which is an $(m-1)$-cube subdivided into $(m-1)!$ simplicies in the standard way:
\[
I^{m-1}= |\underbrace{\Delta^1 \times \dots \times \Delta^1}_{m-1}|
\]
where the vertices of this cube correspond to $2^{m-1}$ monotone edge paths in $\Delta^m$ from $0$ to $m$. 
\end{prop}
\begin{proof}
It is a straightforward to check that the $(m-1)$-simplicies $\Delta_{\sigma}^{m-1}$ have $(m-2)$ out of $m$ faces identified within each other, given by
\[
\{d_i^{m-1}(\Delta_{\sigma}^{m-1})\}_{i=1,\dots,m-2}
\] 
Identifying those faces, we get a closed $(m-1)$-dimensional space that is isomorphic to an $(m-1)$-dimensional cube. Its boundary is given by the $m(m-1)!-(m-2)(m-1)!$=$2(m-1)!$ faces of the simplicies $\Delta_{\sigma}^{m-1}$ that were not identified which are of the type $d_0^{m-1}(\Delta_{\sigma}^{m-1})$ or $d_{m-1}^{m-1}(\Delta_{\sigma}^{m-1})$. It can be rearranged in $2(m-1)$ components each of them having the shape of an $(m-2)$-dimensional cube. Namely, the faces that were not identified can be organized in $2(m-1)$ sets, each of them having $(m-2)!$ elements according to the following equivalence relation
\begin{align*}
&d_0^{m-1}(\Delta_{\sigma}^{m-1})\simeq d_0^{m-1}(\Delta_{\sigma'}^{m-1}) \iff \sigma(1)=i=\sigma'(1), & 1\le i \le m \\
&d_{m-1}^{m-1}(\Delta_{\tau}^{m-1})\simeq d_{m-1}^{m-1}(\Delta_{\tau'}^{m-1}) \iff \tau(m-1)=i=\tau'(m-1), & 1\le i \le m \,
\end{align*}
those $(m-2)$-dimensional simplicies can be glued along $(m-3)(m-2)!$ $(m-3)$-simplicies, in order to get an $(m-2)$-dimensional cube, whose external faces are of  type 
\[
\{d_0^{m-2}(d_0^{m-1})(\Delta_{\sigma}^{m-1}), d_{m-1}^{m-2}(d_0^{m-1})(\Delta_{\sigma}^{m-1})\}
\] 
or of type 
\[
\{d_0^{m-2}(d_{m-1}^{m-1})(\Delta_{\tau}^{m-1}), d_{m-1}^{m-2}(d_{m-1}^{m-1})(\Delta_{\tau}^{m-1})\}
\] 
according to the relation above defined.
\end{proof}
\begin{cor}
Let $n>0$ and $0\le i\le j\le n$. Then the simplicial sets $Map_{\mathcal{C}[\Delta^n]}(i,j)$ are given by
\begin{align*}
&Map_{\mathcal{C}[\Delta^n]}(i,i) \simeq \ast \\
&Map_{\mathcal{C}[\Delta^n]}(i,i+1) \simeq \ast \\
&Map_{\mathcal{C}[\Delta^n]}(i,j) \simeq \coprod_{\sigma \in \Sigma_{j-i-1}} \Delta_{\sigma}^{j-i-1}, & i+1<j  \, 
\end{align*} 
\qed
\end{cor}
\begin{ex}\normalfont
Let us describe $Map_{\mathcal{C}[\Delta^m]}(0,m)$ in the case $m=4$. We have $3!=6$  $3$-simplicies corresponding to the elements of $\Sigma_{3}$
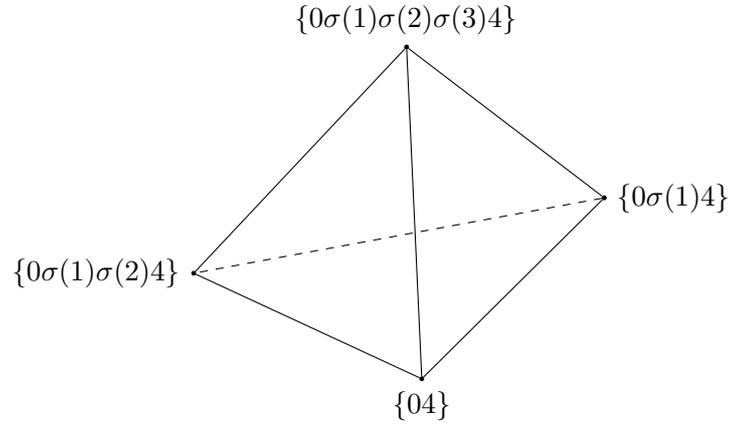
\[
\begin{tikzpicture}
[scale=2, vertices/.style={draw, fill=black, circle, inner sep=0.5pt}]
\node[vertices, label=below:{$\{04\}$}] (a) at (0,0) {};
\node[vertices, label=right:{$\{0\sigma(1)4\}$}] (b) at (1.2,1.2) {};
\node[vertices, label=left:{$\{0\sigma(1)\sigma(2)4\}$}] (c) at (-1.5,0.7) {};
\node[vertices, label=above:{$\{0\sigma(1)\sigma(2)\sigma(3)4\}$}] (d) at (-0.1,2.2) {};
\foreach \to/\from in {a/b,a/c,a/d,b/d,c/d}
\draw [-] (\to)--(\from);
\draw [dashed] (b)--(c);
\end{tikzpicture}
\]
and we have the identifications of the $2$-faces of those $3$-simplicies 
\begin{align*}
&d_1^3(\Delta_{Id}^3) \simeq d_1^3(\Delta_{(12)}^3) \\
&d_2^3(\Delta_{Id}^3) \simeq d_2^3(\Delta_{(23)}^3) \\
&d_2^3(\Delta_{(12)}^3) \simeq d_2^3(\Delta_{(123)}^3) \\
&d_1^3(\Delta_{(23)}^3) \simeq d_1^3(\Delta_{(132)}^3) \\
&d_1^3(\Delta_{(13)}^3) \simeq d_1^3(\Delta_{(123)}^3) \\
&d_2^3(\Delta_{(13)}^3) \simeq d_2^3(\Delta_{(132)}^3) \, 
\end{align*} 
\newpage
after this identification, we obtain a $3$-dimensional cube 
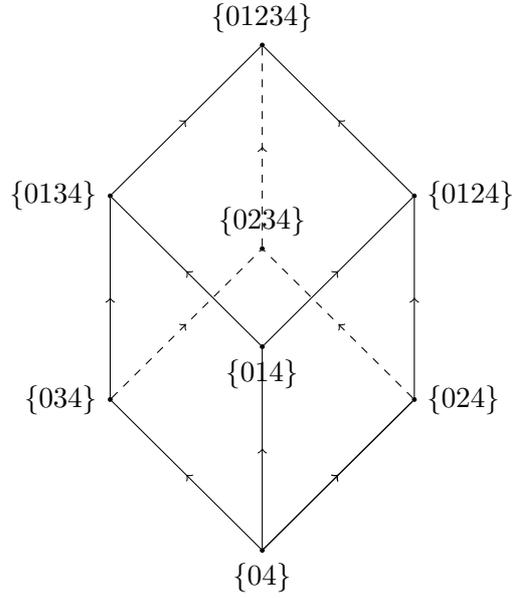
\[
\begin{tikzpicture}
[scale=1, vertices/.style={draw, fill=black, circle, inner sep=0.5pt}]
\node[vertices, label=below:{$\{014\}$}] (a) at (0,2.7) {};
\node[vertices, label=right:{$\{024\}$}] (b) at (2,2) {};
\node[vertices, label=left:{$\{034\}$}] (c) at (-2,2) {};
\node[vertices, label=below:{$\{04\}$}] (d) at (0,0) {};
\node[vertices, label=left:{$\{0134\}$}] (e) at (-2,4.7) {};
\node[vertices, label=right:{$\{0124\}$}] (f) at (2,4.7) {};
\node[vertices, label=above:{$\{0234\}$}] (g) at (0,4) {};
\node[vertices, label=above:{$\{01234\}$}] (h) at (0,6.7) {};
\begin{scope}[decoration={markings,mark=at position 0.5 with {\arrow{>}}}] 
\draw[postaction={decorate}] (d) -- node[auto,swap] {$$} (a);
\draw[postaction={decorate}] (d) -- node[auto] {$$} (b);
\draw[postaction={decorate}] (d) -- node[auto,swap] {$$} (c);
\draw[postaction={decorate}] (c) -- node[auto] {$$} (e);
\draw[postaction={decorate},dashed] (c) -- node[auto] {$$} (g);
\draw[postaction={decorate}] (b) -- node[auto] {$$} (f);
\draw[postaction={decorate},dashed] (b) -- node[auto] {$$} (g);
\draw[postaction={decorate}] (a) -- node[auto] {$$} (e);
\draw[postaction={decorate}] (a) -- node[auto] {$$} (f);
\draw[postaction={decorate}] (e) -- node[auto] {$$} (h);
\draw[postaction={decorate},dashed] (g) -- node[auto] {$$} (h);
\draw[postaction={decorate}] (f) -- node[auto] {$$} (h);
\draw[postaction={decorate}] (d) -- node[auto,swap] {$$} (b);
\end{scope}
\end{tikzpicture}
\]
whose $2$-dimensional cubical faces are
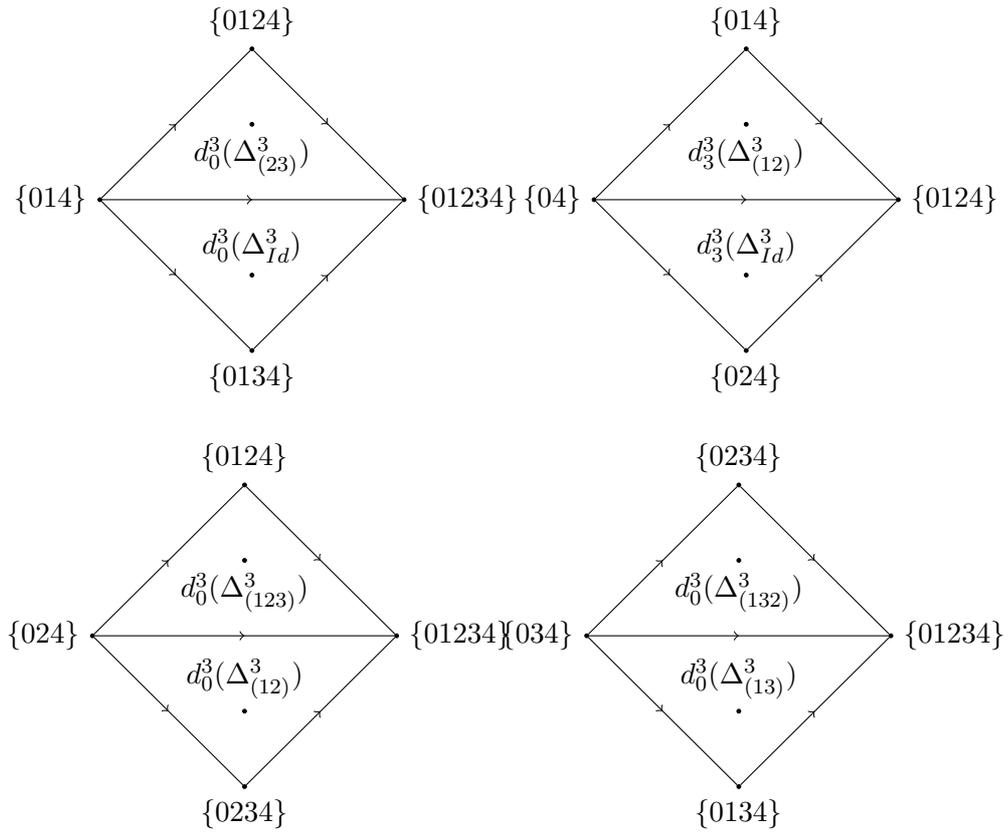
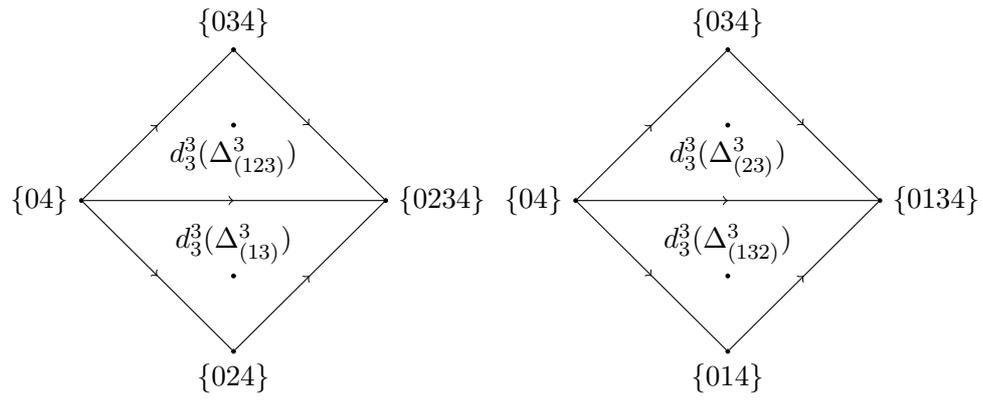
\[
\begin{tikzpicture}
[scale=1, vertices/.style={draw, fill=black, circle, inner sep=0.5pt}]
\node[vertices, label=above:{$\{0124\}$}] (a) at (0,4) {};
\node[vertices, label=right:{$\{01234\}$}] (b) at (2,2) {};
\node[vertices, label=left:{$\{014\}$}] (c) at (-2,2) {};
\node[vertices, label=below:{$\{0134\}$}] (d) at (0,0) {};
\node[vertices, label=above:{$d_0^3(\Delta_{Id}^3)$}] (p) at (0,1) {};
\node[vertices, label=below:{$d_0^3(\Delta_{(23)}^3)$}] (q) at (0,3) {};
\node[vertices, label=above:{$\{014\}$}] (e) at (6.5,4) {};
\node[vertices, label=right:{$\{0124\}$}] (f) at (8.5,2) {};
\node[vertices, label=left:{$\{04\}$}] (g) at (4.5,2) {};
\node[vertices, label=below:{$\{024\}$}] (h) at (6.5,0) {};
\node[vertices, label=above:{$d_3^3(\Delta_{Id}^3)$}] (s) at (6.5,1) {};
\node[vertices, label=below:{$d_3^3(\Delta_{(12)}^3)$}] (t) at (6.5,3) {};
\begin{scope}[decoration={markings,mark=at position 0.5 with {\arrow{>}}}] 
\draw[postaction={decorate}] (c) -- node[auto,swap] {$$} (a);
\draw[postaction={decorate}] (a) -- node[auto] {$$} (b);
\draw[postaction={decorate}] (c) -- node[auto,swap] {$$} (b);
\draw[postaction={decorate}] (c) -- node[auto] {$$} (d);
\draw[postaction={decorate}] (d) -- node[auto,swap] {$$} (b);
\draw[postaction={decorate}] (g) -- node[auto,swap] {$$} (e);
\draw[postaction={decorate}] (e) -- node[auto] {$$} (f);
\draw[postaction={decorate}] (g) -- node[auto,swap] {$$} (f);
\draw[postaction={decorate}] (g) -- node[auto] {$$} (h);
\draw[postaction={decorate}] (h) -- node[auto,swap] {$$} (f);
\end{scope}
\end{tikzpicture}
\]
\[
\begin{tikzpicture}
[scale=1, vertices/.style={draw, fill=black, circle, inner sep=0.5pt}]
\node[vertices, label=above:{$\{0124\}$}] (a) at (0,4) {};
\node[vertices, label=right:{$\{01234\}$}] (b) at (2,2) {};
\node[vertices, label=left:{$\{024\}$}] (c) at (-2,2) {};
\node[vertices, label=below:{$\{0234\}$}] (d) at (0,0) {};
\node[vertices, label=above:{$d_0^3(\Delta_{(12)}^3)$}] (p) at (0,1) {};
\node[vertices, label=below:{$d_0^3(\Delta_{(123)}^3)$}] (q) at (0,3) {};
\node[vertices, label=above:{$\{0234\}$}] (e) at (6.5,4) {};
\node[vertices, label=right:{$\{01234\}$}] (f) at (8.5,2) {};
\node[vertices, label=left:{$\{034\}$}] (g) at (4.5,2) {};
\node[vertices, label=below:{$\{0134\}$}] (h) at (6.5,0) {};
\node[vertices, label=above:{$d_0^3(\Delta_{(13)}^3)$}] (s) at (6.5,1) {};
\node[vertices, label=below:{$d_0^3(\Delta_{(132)}^3)$}] (t) at (6.5,3) {};
\begin{scope}[decoration={markings,mark=at position 0.5 with {\arrow{>}}}] 
\draw[postaction={decorate}] (c) -- node[auto,swap] {$$} (a);
\draw[postaction={decorate}] (a) -- node[auto] {$$} (b);
\draw[postaction={decorate}] (c) -- node[auto,swap] {$$} (b);
\draw[postaction={decorate}] (c) -- node[auto] {$$} (d);
\draw[postaction={decorate}] (d) -- node[auto,swap] {$$} (b);
\draw[postaction={decorate}] (g) -- node[auto,swap] {$$} (e);
\draw[postaction={decorate}] (e) -- node[auto] {$$} (f);
\draw[postaction={decorate}] (g) -- node[auto,swap] {$$} (f);
\draw[postaction={decorate}] (g) -- node[auto] {$$} (h);
\draw[postaction={decorate}] (h) -- node[auto,swap] {$$} (f);
\end{scope}
\end{tikzpicture}
\]
\[
\begin{tikzpicture}
[scale=1, vertices/.style={draw, fill=black, circle, inner sep=0.5pt}]
\node[vertices, label=above:{$\{034\}$}] (a) at (0,4) {};
\node[vertices, label=right:{$\{0234\}$}] (b) at (2,2) {};
\node[vertices, label=left:{$\{04\}$}] (c) at (-2,2) {};
\node[vertices, label=below:{$\{024\}$}] (d) at (0,0) {};
\node[vertices, label=above:{$d_3^3(\Delta_{(13)}^3)$}] (p) at (0,1) {};
\node[vertices, label=below:{$d_3^3(\Delta_{(123)}^3)$}] (q) at (0,3) {};
\node[vertices, label=above:{$\{034\}$}] (e) at (6.5,4) {};
\node[vertices, label=right:{$\{0134\}$}] (f) at (8.5,2) {};
\node[vertices, label=left:{$\{04\}$}] (g) at (4.5,2) {};
\node[vertices, label=below:{$\{014\}$}] (h) at (6.5,0) {};
\node[vertices, label=above:{$d_3^3(\Delta_{(132)}^3)$}] (s) at (6.5,1) {};
\node[vertices, label=below:{$d_3^3(\Delta_{(23)}^3)$}] (t) at (6.5,3) {};
\begin{scope}[decoration={markings,mark=at position 0.5 with {\arrow{>}}}] 
\draw[postaction={decorate}] (c) -- node[auto,swap] {$$} (a);
\draw[postaction={decorate}] (a) -- node[auto] {$$} (b);
\draw[postaction={decorate}] (c) -- node[auto,swap] {$$} (b);
\draw[postaction={decorate}] (c) -- node[auto] {$$} (d);
\draw[postaction={decorate}] (d) -- node[auto,swap] {$$} (b);
\draw[postaction={decorate}] (g) -- node[auto,swap] {$$} (e);
\draw[postaction={decorate}] (e) -- node[auto] {$$} (f);
\draw[postaction={decorate}] (g) -- node[auto,swap] {$$} (f);
\draw[postaction={decorate}] (g) -- node[auto] {$$} (h);
\draw[postaction={decorate}] (h) -- node[auto,swap] {$$} (f);
\end{scope}
\end{tikzpicture}
\]
\end{ex}
\newpage
\subsection{The big dg-nerve and comparison with the small dg-nerve} 
Given a simplicial category $\mathcal{C}$, its simplicial nerve is the simplicial set $N_{SCat}(\mathcal{C})$ whose $n$-simplicies are determined by the adjunction (\ref{adj})
\[
N_{SCat}(\mathcal{C})_{n}=Hom_{SCat}(\mathcal{C}[\Delta^n], \mathcal{C})
\]
Hence, an $n$-simplex of $N_{SCat}(\mathcal{C})$ is a $\mathcal{C}[\Delta^n]$ shaped simplicial diagram in $\mathcal{C}$. In particular
\[
N_{SCat}(\mathcal{C})_{0}=Ob(\mathcal{C})
\]
and for $n>0$ an element of $N_{SCat}(\mathcal{C})_{n}$ is given by the list of data 
\[
 \begin{cases} x_0, x_1, \dots , x_n \in Ob(\mathcal{C}) \\ 
 g_{I_0}  \in Map_{\mathcal{C}}(x_i, x_j)_0, I_0 \subseteq \{i,\dots , j\}, max(I_0)=i, min(I_0)=j\\ 
 g_{I_0,I_1}  \in Map_{\mathcal{C}}(x_i, x_j)_1, I_0 \hookrightarrow I_1 \subseteq \{i,\dots , j\}\\ 
 \dots \\
 g_{I_0,I_1,\dots,I_{n-1}}  \in Map_{\mathcal{C}}(x_i, x_j)_1, I_0 \hookrightarrow \dots \hookrightarrow I_{n-1} \subseteq \{i,\dots , j\}\\ 
\end{cases}
\]  
where, fixed $0\le i<j\le n$, $I_p$ is an ordered subset of $\{i,\dots , j\}$ with minimum $i$ and maximum $j$. Those data have to satisfy to the relations
\begin{align*}
& d^{k}_j(g_{I_0, \dots, I_k})=g_{I_0, \dots, \hat{I_j}, \dots, I_{k}}, & 0\le k\le n-1, 0\le j\le k \\
&s^{k}_j(g_{I_0, \dots, I_k})=g_{I_0, \dots, I_j, I_j, \dots, I_{k}}, & 0\le k\le n-1, 0\le j< k \,
\end{align*}
where $d^{k}_j$ and $s^{k}_j$ are the simplicial maps of the simplicial category $\mathcal{C}$, and the relations given by the composition
\[
g_{I_0 \cup J_0, \dots, I_k \cup J_k} = g_{I_0, \dots, I_k} \circ g_{J_0, \dots, J_k} 
\]
\vspace{5 mm}
\begin{defi}[Big dg-nerve, \cite{Lurie 2}]
Let $\mathcal{D}$ be a dg-category, the big dg-nerve is the simplicial nerve of the simplicial category $\mathcal{D}_{\Delta}$
\[
N_{dg}^{big}(\mathcal{D}):=N_{SCat}(\mathcal{D}_{\Delta})
\]
\end{defi}
The simplicial category $\mathcal{D}_{\Delta}$ is a fibrant simplicial category because the mapping spaces are the underlying simplicial sets of a simplicial object in the abelian category $Mod_{\mathbb{K}}$, and it is a general fact that those are always Kan complexes \cite{Wei}. In particular simplicial nerves of fibrant simplicial categories are always $\infty$-category \cite{Lurie 1} and so it is the big dg-nerve. We describe in the following proposition how to construct an equivalence of $\infty$-categories between $N_{dg}^{big}(\mathcal{D})$ and $N_{dg}^{sm}(\mathcal{D})$ using the cubical interpretation of the simplicial nerve of a simplicial category given in the previous section.
\newpage
\begin{prop}
The big and the small nerve of a dg-category are equivalent $\infty$-categories.
\begin{proof}
Let $\mathcal{D}$ be a dg-category and $n\ge0$. Define the map of sets
\[
N_{dg}^{big}(\mathcal{D})_n \to N_{dg}^{sm}(\mathcal{D})_n
\] 
in the following way: let $l\ge 0$, consider 
\[
\pi_l : Map_{\mathcal{D}_{\Delta}}(x_{i_0},x_{i_k})_l \to Hom^{-l}_{\mathcal{D}}(x_{i_0},x_{i_k})
\] 
the projection onto the highest degree component given by the construction of the mapping space of $\mathcal{D}_{\Delta}$ via Dold-Kan. Given an $n$-simplex of $N_{dg}^{big}(\mathcal{D})=N_{SCat}(\mathcal{D}_{\Delta})$
\[
(\{x_i\}^n_{i=0},\{g_{I_0}\}, \dots, \{g_{I_0,I_1,\dots,I_{n-1}} \}_{I_0 \hookrightarrow \dots \hookrightarrow I_{n-1}} )
\] 
we produce an $n$-simplex of $N_{dg}^{sm}(\mathcal{D})$ 
\[
(\{x_{i_0}\}, \{f_{i_0i_1}\}, \dots, \{f_{i_0\dots i_n}\}))
\] 
using the following inductive definition: objects $\{x_{i_0}\}$ are the same. For $0\le i_0<i_1<i_2\le n$
\[
f_{i_0i_1}=g_{\{i_0,i_1\}}
\]
\[
f_{i_0i_1i_2}=-\pi_1(g_{\{i_0,i_2\} \hookrightarrow \{i_0,i_1,i_2\}}) 
\]
Now, assume we defined $f_{i_0i_1\dots i_{k-1}}$ for any $0\le i_0<i_1<\dots < i_{k-1}\le n $ then we define $f_{i_0i_1\dots i_k}$ for a fixed $0\le i_0<i_1<\dots < i_k\le n$. Let $\sigma \in \Sigma_{k-1}$, consider the chain of inclusions 
\[
\{i_0,i_k\} \hookrightarrow \{i_0,i_{\sigma(1)},i_k\} \hookrightarrow \{i_0,i_{\sigma(1)}, i_{\sigma(2)},i_k\} \hookrightarrow \dots \hookrightarrow \{i_0,i_{\sigma(1)}, \dots,i_k\}
\] 
and let $g_{\sigma}=g_{\{i_0,i_k\},\{i_0,i_{\sigma(1)},i_k\},\dots,\{i_0,i_{\sigma(1)}, \dots,i_k\}}\in Map_{\mathcal{D}_{\Delta}}(x_{i_0},x_{i_k})_{k-1}$. Define
\[
f_{i_0i_1\dots i_k}=(-1)^{k-1}\sum_{\sigma \in \Sigma_{k-1}} sgn(\sigma) \pi_{k-1}(g_{\sigma})
\]
this is a well defined element of $Hom^{1-k}_{\mathcal{D}}(x_{i_0},x_{i_k})$. We want to show that the its boundary is given by the formula (\ref{dgeq}). According to (\ref{DK}), we have that
\[
d(\pi_{k-1}(g_{\sigma}))=\sum_{j=0}^{k-1}(-1)^j\pi_{k-2}(d_j^{k-1}(g_{\sigma}))
\]
The description via cubes induces a simplification of the terms of the form $\pi_{k-2}(d_j^{k-1}(g_{\sigma}))$ for $0< j <k-1$. We find hence that 
\[
d(f_{i_0i_1\dots i_k})=\sum_{1\le j\le k-1} \sum_{\sigma(1)=j} (-1)^{k-1}sgn(\sigma) \pi_{k-2}(d_0^{k-1}(g_{\sigma})) + 
\]
\[
+\sum_{1\le j\le k-1} \sum_{\tau(k-1)=j} sgn(\tau) \pi_{k-2}(d_{k-1}^{k-1}(g_{\tau}))
\]
Consider the equivalence class of permutations $\tau$, with $\tau(k-1)=k-1$. There are exactly $(k-2)!$ permutation of this type and this is a subgroup of $\Sigma_{k-1}$ isomorphic to $\Sigma_{k-2}$. If $\tau'$ is the permutation in $\Sigma_{k-2}$ associated to $\tau \in \Sigma_{k-1}$ under this isomorphism, we have that
\[
d_{k-1}^{k-1}(g_{\tau})= g_{\{i_0i_k\},\dots,\{i_0,i_{\tau(1)}, \dots,i_{\tau(k-2)},i_k\}}=g_{\{i_0i_k\},\dots,\{i_0,i_{\tau'(1)}, \dots,i_{\tau'(k-1)},i_k\}}
\]
and hence
\[ 
\sum_{\tau(k-1)=k-1} sgn(\tau) \pi_{k-2}(d_{k-1}^{k-1}(g_{\tau}))=\sum_{\tau' \in \Sigma_{k-2}} sgn(\tau') \pi_{k-2}(g_{\{i_0i_k\},\dots,\{i_0,i_{\tau'(1)}, \dots,i_{\tau'(k-1)},i_k\}})=
\]
\[
=(-1)^{k-2} f_{i_0i_1\dots \hat{i}_{k-1} i_k}
\]
and for a general equivalence $\tau(k-1)=j$, multiplication on the right by the permutation $(k-1\dots j)$, induces a bijection with the equivalence class $\tau'(k-1)=k-1$, hence we find
\[
\sum_{\tau(k-1)=j} sgn(\tau) \pi_{k-2}(d_{k-1}^{k-1}(g_{\tau}))=\sum_{\tau'(k-1)=1} (-1)^{k-1-j} sgn(\tau') \pi_{k-2}(d_{k-1}^{k-1}(g_{\tau' \circ (j\dots k-1)}))=
\]
\[
=(-1)^{j-1} f_{i_0\dots \hat{i}_{j} \dots i_k}
\]
Consider now the class of permutations $\sigma \in \Sigma_{k-1}$ with $\sigma(1)=j$. Given such $\sigma$, define
\[
\alpha_{\sigma}=\{p\mid \sigma(p)>j, 2\le p\le k-1\}=\{p_1<\dots < p_n\}
\]
\[
\beta_{\sigma}=\{q\mid \sigma(q)<j, 2\le q\le k-1\}=\{q_1<\dots < q_m\}
\]
to those, we can associate the $n$-simplex
\[
g_{\{i_0i_j\}, \{i_0i_{\sigma(p_1)}i_j\},\dots,\{i_0,i_1, \dots ,i_j\}}
\]
and the $m$-simplex
\[
g_{\{i_ji_k\}, \{i_0i_{\sigma(q_1)}i_j\},\dots,\{i_j,i_{j+1}, \dots ,i_k\}}
\]
It is easy to see that
\[
d_0^{k-1}(g_{\sigma})=s_{\bar{p}}(g_{\{i_ji_k\}, \{i_ji_{\sigma(q_1)}i_k\},\dots,\{i_j,i_{j+1}, \dots ,i_k\}})\circ s_{\bar{q}}(g_{\{i_0i_j\}, \{i_0i_{\sigma(p_1)}i_j\},\dots,\{i_0,i_1, \dots ,i_j\}}) 
\]
where
\[
s_{\bar{q}}=s_{q_m-2}\circ \dots \circ s_{q_1-2} 
\]
\[
s_{\bar{p}}=s_{p_n-2}\circ \dots \circ s_{p_1-2}
\]
By the definition of the Alexander-Whitney map (\ref{AW}), we have that $\pi_{k-2}(d_0^{k-1}(g_{\sigma}))$ equals to
\[
\sum_{p+q=k-2}(-1)^{(k-2)p+1}\pi_q(\tilde{d}^p(s_{\bar{p}}(g_{\{i_ji_k\}, \dots,\{i_j,\dots ,i_k\}})))\circ \pi_p(d_0^q(s_{\bar{q}}(g_{\{i_0i_j\},\dots,\{i_0, \dots ,i_j\}})))
\] 
but 
\[
\pi_q(\tilde{d}^p(s_{\bar{p}}(g_{\{i_ji_k\}, \dots,\{i_j,\dots ,i_k\}})))
\]
is non zero if and only if $p\ge (k-2)-m$ and similarly
\[
\pi_p(d_0^q(s_{\bar{q}}(g_{\{i_0i_j\},\dots,\{i_0, \dots ,i_j\}})))
\]
is non zero if and only if $q\ge m$ and because $p+q=k-2$, both are non zero only when $p=(k-2)-m$ and $q=m$. Hence we have that  $\pi_{k-2}(d_0^{k-1}(g_{\sigma}))$ equals to
\[
(-1)^{(k-2)n+1}\pi_m(\tilde{d}^n(s_{\bar{p}}(g_{\{i_ji_k\}, \dots,\{i_j,\dots ,i_k\}})))\circ \pi_n(d_0^m(s_{\bar{q}}(g_{\{i_0i_j\},\dots,\{i_0, \dots ,i_j\}})))
\]
Now, those components may still be zero, depending on $\sigma$. In particular the left hand side is non zero if and only if
\[
p_n=k-1, p_{n-1}=k-2, \dots, p_1=k-n
\]
and the right hand side is non zero if and only if
\[
q_m=2, q_{m-1}=3, \dots, q_1=m+1
\]
Those conditions impose that $\sigma$ has to satisfy
\begin{align*}
&\sigma(1)=j \\
&\sigma(\{2,\dots,m+1\})=\{j+1,\dots,k-1\} \\
&\sigma(\{m+2,\dots,k-1\})=\{1,\dots,j-1\} \, 
\end{align*} 
In particular $m=k-j-1$ and $n=j-1$. Now, taking restrictions of $\sigma$ to those subsets, we can associate to it two permutations, $\tau \in \Sigma_{j-1}$ and $\rho \in \Sigma_{k-j-1}$. This association is bijective and
\[
sgn(\sigma)=-sgn(\tau)sgn(\rho)
\]
Hence we get the final equation given by
\[
\sum_{\sigma(1)=j} (-1)^{k-1}sgn(\sigma) \pi_{k-2}(d_0^{k-1}(g_{\sigma}))= 
\]
\[
=(-1)^{k+k(j-1)+1}(\sum_{\rho \in \Sigma_{k-j-1}} sgn(\rho) \pi_{k-j-1}(g_{\rho}))\circ (\sum_{\tau \in \Sigma_{j-1}} sgn(\tau) \pi_{k-j-1}(g_{\tau}))=
\]
\[
=(-1)^{k(j-1)+1} f_{i_j\dots i_k}\circ f_{i_0\dots i_j}
\]
Hence we found that
\[
d(f_{i_0\dots i_k})=\sum_{0<j<n}(-1)^{j-1} f_{i_0\dots \hat{i_j}\dots i_k} + \sum_{0<j<n}(-1)^{1+k(j-1)}f_{i_j\dots i_k}\circ f_{i_0\dots i_j}
\]
This construction clearly respects the simplicial structure and hence defines a morphism of simplicial sets $N_{dg}^{big}(\mathcal{D}) \to N_{dg}^{sm}(\mathcal{D})$. According to J.Lurie \cite{Lurie 2}, this morphism is an equivalence of $\infty$-categories.
\end{proof}
\end{prop}
\newpage
\section{Dg-nerve of pretriangulated dg-categories}
We recall in this section the definition of a pretriangulated dg-category and the one of stable $\infty$-category, the notion of homotopy limits and colimits in model categories and how to compute them when the indexing category is Reedy. We prove that $N^{big}_{dg}(\mathcal{D})$, and hence $N^{sm}_{dg}(\mathcal{D})$,  of a pretriangulated dg-category $\mathcal{D}$ is a stable $\infty$-category.
\subsection{Pretriangulated dg-categories and stable $\infty$-categories}
\begin{defi}[dg-category of twisted complexes, \cite{MK}]
Let $\mathcal{D}$ be a dg-category. A twisted complex in $\mathcal{D}$ is the data consisting of a pair $K=(K_i, q_{ij})_{i,j\in \mathbb{Z}}$, where:
\vspace{3 mm}
\begin{itemize}
\item $K_i\in Ob(\mathcal{D})$ are $0$ for almost all $i\in \mathbb{Z}$
\item $q_{ij}\in Hom^{i-j+1}_{\mathcal{D}}(K_i,K_j)$ such that 
\[
d(q_{ij})+\sum_{k\in \mathbb{Z}} q_{kj}q_{ik}=0
\]
\end{itemize}
The dg-category $PreTr(\mathcal{D})$ of twisted complexes over $\mathcal{D}$ is the differential category whose objects are twisted complexes over $\mathcal{D}$ and cochain complex of morphisms
\[
Hom^{k}_{PreTr(\mathcal{D})}(K,K')=\underset{l+j-i=k}{\bigoplus}Hom^{l}_{\mathcal{D}}(K_i,K'_j)
\]
with differential $d$ given for $f\in Hom^{l}_{\mathcal{D}}(K_i,K'_j)$ by
\[
d(f)=d(f)+\sum_m (q'_{jm}f+(-1)^{l(i-m+1)}fq_{mi})
\]
\end{defi}
\begin{rem}\normalfont
The dg-category $PreTr(\mathcal{D})$ is the "smallest" dg category in which $\mathcal{D}$ embeds and for which it is possible to define the shift functor and functorial cones. 

To make this precise, consider the category of dg-functors $dgFun(\mathcal{D},Ch^{\bullet}(Vect_{\mathbb{K}}))$. This category as a dg-enrichment, as described in [Ex.\ref{ex1}], with respect to which the diagram of dg-functors
 \[
  \begin{tikzpicture}
    \def\x{1.5}
    \def\y{-1.2}
    \node (A1_1) at (0*\x, 2*\y) {$\mathcal{D}$};
    \node (A2_2) at (3*\x, 2*\y) {$dgFun(\mathcal{D},Ch^{\bullet}(Vect_k))$};
    \node (A1_2) at (1.5*\x, 3.5*\y) {$PreTr(\mathcal{D})$};
   \path (A1_1) edge [right hook->] node [auto] {$\scriptstyle{\epsilon}$} (A1_2);
    \path (A1_2) edge [right hook->] node [auto] {$\scriptstyle{\alpha}$} (A2_2);
    \path (A1_1) edge [right hook->] node [auto] {$\scriptstyle{h}$} (A2_2);
     \end{tikzpicture}
  \]
is commutative, where $h$ is the dg-Yoneda functor
\[
h(X)(Y)=Hom^{\bullet}_{C}(X,Y)
\]
$\epsilon$ is the dg-functor sending an object $X\in Ob(C)$ to the twisted complex concetrated in degree $0$ and $\alpha$ is the dg-functor that to a twisted complex $K=(K_i, q_{ij})_{i,j\in \mathbb{Z}}$ associates the dg-functor
\[
\alpha(K)(Y)=\bigoplus_{i\in \mathbb{Z}} Hom^{\bullet}_{C}(Y,K_i)[-i]
\]
with twisted differential $d + q$. Moreover, both dg-categories $PreTr(\mathcal{D})$ and $dgFun(\mathcal{D},Ch^{\bullet}(Vect_{\mathbb{K}}))$ have shift functors and functorial cones which are preserved by the dg-functor $\alpha$. Namely, on the dg-category $dgFun(\mathcal{D},Ch^{\bullet}(Vect_{\mathbb{K}}))$ we define the shift according to the formula (\ref{sh}). If $\eta: F\to G$ is a dg-natural transformation of dg-functors, then its cone $Cone(\eta)$ is defined objectwise by 
\[
Cone(\eta)(X)=Cone(\eta(X))
\] 
Similarly, if $K$ is a twisted complex, its shift by $1$ is given by
\[
K[1]_i=K_{i+1}
\] 
\[
q[1]_{ij}=q_{i+1,j+1}
\]
If $f:K\to K'$ is a morphism of twisted complexes, where $K=(K_i, q_{ij})$ and $K'=(K'_i, q'_{ij})$, its cone is given by 
\[
Cone(f)=(K_{i+1}\bigoplus K'_i, q''_{ij})
\]
where $q''_{ij}$ is given by the matrix
\[
q''_{ij}=\begin{vmatrix} q_{i+1,j+1} & f_{i+1,j} \\ 0 & q'_{ij} \end{vmatrix}
\]
Those constructions on the dg-category $PreTr(\mathcal{D})$ are compatible with the dg-functor $\alpha$. The remarkable fact is that shift and cones so defined on $PreTr(\mathcal{D})$ induce a triangulated structure on the category $H^0(PreTr(\mathcal{D}))$ \cite{MK}, where exact triangles are given by sequences in $H^0(PreTr(\mathcal{D}))$
\[
\begin{tikzpicture}
    \def\x{1.5}
    \def\y{-1.2}
    \node (A1_1) at (0*\x, 2*\y) {$K$};
    \node (A2_2) at (1*\x, 2*\y) {$K'$};
    \node (A3_3) at (2.2*\x, 2*\y) {$Cone(f)$};
   \node (A4_4) at (3.7*\x, 2*\y) {$K[1]$};
   \path (A1_1) edge [->] node [auto] {$\scriptstyle{f}$} (A2_2);
    \path (A2_2) edge [->] node [auto] {$\scriptstyle{}$} (A3_3);
    \path (A3_3) edge [->] node [auto] {$\scriptstyle{}$} (A4_4);
     \end{tikzpicture}
\]
\end{rem}
We recall now the definition of pretriangulated dg-category.
\begin{defi}[Pretriangulated dg-category, \cite{MK}]
A dg-category $\mathcal{D}$ is called pretriangulated if, for every twisted complex $K\in PreTr(\mathcal{D})$, the dg-functor $\alpha(K)$ is representable, i.e. it is isomorphic, as a dg-functor, to $h(X)$ for some object $X\in \mathcal{D}$.
\end{defi}
A consequence of this definition is the following proposition \cite{MK}.
\begin{prop}
If $\mathcal{D}$ is a pretriangulated dg-category, then the dg-functor $\epsilon: \mathcal{D}\to PreTr(\mathcal{D})$ is a quasi-equivalence of dg-categories. 
\end{prop}
\vspace{3 mm}
This proposition allows to transfer the shift and the cone construction to any pretriangulated dg-category $\mathcal{D}$. Namely, let $T$ the inverse equivalence of $\epsilon$, then for $X\in Ob(C)$ define its shift by $+1$ by
\begin{equation}\label{sh1}
X[1]=T(\epsilon(X)[1])
\end{equation}
where $T(\epsilon(X)[1])$ is an object of $\mathcal{D}$ representing the functor $\alpha(T(\epsilon(X)[1]))$. Similarly, if $f\in Hom^{\bullet}_{\mathcal{D}}(X,Y)$ of degree $0$, then its cone is given by 
\[
Cone(f)=T(Cone(\epsilon(f)))
\]
As twisted complexes those are given by
\[
\epsilon(X)[1]_i=
\begin{cases}
X, & \text{if $i=-1$} \\
0, & \text{otherwise}
\end{cases}
\] 
and $Cone(\epsilon(f))$ is given by
\[
 \begin{tikzpicture}
    \def\x{3}
    \def\y{-1.5}
    \node (A2_2) at (1.7*\x, 2*\y) {$(Y)_0$};
    \node (A2_1) at (1*\x, 2*\y) {$(X)_{-1}$};
   \path (A2_1) edge [->] node [auto] {$\scriptstyle{f}$} (A2_2);
    \end{tikzpicture}
\]
We have the following proposition \cite{MK}.
\begin{prop}
Let $\mathcal{D}$ a pretriangulated dg-category, then its $0$-th cohomology category $H^0(\mathcal{D})$ is triangulated with the shift functor defined in \textnormal{(\ref{sh1})}
\[
[1]: H^0(\mathcal{D})\to H^0(\mathcal{D})
\]
and class of exact triangles of the form
\[
\begin{tikzpicture}
    \def\x{1.5}
    \def\y{-1.2}
    \node (A1_1) at (0*\x, 2*\y) {$K$};
    \node (A2_2) at (1*\x, 2*\y) {$K'$};
    \node (A3_3) at (2.2*\x, 2*\y) {$Cone(f)$};
   \node (A4_4) at (3.7*\x, 2*\y) {$K[1]$};
   \path (A1_1) edge [->] node [auto] {$\scriptstyle{f}$} (A2_2);
    \path (A2_2) edge [->] node [auto] {$\scriptstyle{}$} (A3_3);
    \path (A3_3) edge [->] node [auto] {$\scriptstyle{}$} (A4_4);
     \end{tikzpicture}
\]
where $f$ is a closed morphism of degree $0$. Moreover the functor
\[
H^0(T): H^0(PreTr(\mathcal{D}))\to H^0(\mathcal{D})
\]
is an equivalence of triangulated categories.
\end{prop}
Let $\mathcal{D}$ be a pretriangulated dg-category. Then there are quasi-isomorphims of complexes
\[
Hom^{\bullet}_{\mathcal{D}}(X[1],Y)\simeq Hom^{\bullet}_{\mathcal{D}}(X,Y[-1])\simeq Hom^{\bullet}_{\mathcal{D}}(X,Y)[-1]
\]
and, if $f: X \to Y$ is a closed degree 0 morphism and $Z$ is any other object of $\mathcal{D}$ then 
\[
Hom^k_{\mathcal{D}}(Cone(f),Z)=Hom^k_{\mathcal{D}}(Y,Z)\oplus Hom^{k-1}_{\mathcal{D}}(X,Z)
\]
with differential
\[
d^k=\begin{vmatrix} d^k_{(Y,Z)} & 0 \\ (-1)^k(-\circ f) & d^{k-1}_{(X,Z)} \end{vmatrix}
\]
where $d^k_{(X,Y)}=d_{Hom ^k_{\mathcal{D}}(X,Y)}$, for any $X,Y\in \mathcal{D}$. Similarly
\[
Hom ^k_{\mathcal{D}}(Z,Cone(f))=Hom ^k_{\mathcal{D}}(Z,Y)\oplus Hom^{k+1}_{\mathcal{D}}(Z,X)
\]
with differential
\[
d^k=\begin{vmatrix} d^k_{(Z,Y)} & (-\circ f) \\ 0 & d^{k+1}_{(Z,X)} \end{vmatrix}
\]
\newpage
We recall now the notion of stable $\infty$-category.
\begin{defi} [Pointed $\infty$-category]
An $\infty$-category $X$ is pointed if it has a zero object, that is an object $0 \in X_0$ such that, for every object $A\in X_0$
\[
Map_{X}(A,0)\simeq \ast \simeq Map_{X}(0,A)
\]
\end{defi}
 \begin{defi}[Fiber and cofiber sequence]
Given a pointed $\infty$-category $X$, a triangle is a diagram $\Delta^1 \times \Delta^1\to X$ of the form 
\[
 \begin{tikzpicture}
    \def\x{1.5}
    \def\y{-1.2}
    \node (A2_2) at (2*\x, 2*\y) {$C$};
    \node (A2_1) at (2*\x, 1*\y) {$B$};
    \node (A1_2) at (1*\x, 2*\y) {$0$};
   \node (A1_1) at (1*\x, 1*\y) {$A$};
   \path (A2_1) edge [->] node [auto] {$\scriptstyle{g}$} (A2_2);
    \path (A1_2) edge [->] node [auto,swap] {$\scriptstyle{}$} (A2_2);
       \path (A1_1) edge [->] node [auto] {$\scriptstyle{f}$} (A2_1);
   \path (A1_1) edge [->] node [auto] {$\scriptstyle{}$} (A1_2);

    \end{tikzpicture}
\]
We say that a triangle is a fiber sequence, the fiber of $g$, if it it homotopy cartesian and, dually, a cofiber sequence, the cofiber of $f$, if it is homotopy cocartesian.
\end{defi}
\begin{defi}[Stable $\infty$-category, \cite{Lurie 2}]
An $\infty$-category $X$ is stable if
\begin{itemize}
\item $X$ is pointed
\item Every morphism admits fiber and cofiber
\item A triangle is a fiber sequence iff it is a cofiber sequence
\end{itemize}
\end{defi}
\begin{rem}\normalfont
Let $X$ be a stable $\infty$-category, then for an object $A$ we can consider the diagram
\[
 \begin{tikzpicture}
    \def\x{1.5}
    \def\y{-1.2}
    \node (A2_2) at (2*\x, 2*\y) {$A$};
    \node (A2_1) at (2*\x, 1*\y) {$0$};
    \node (A1_2) at (1*\x, 2*\y) {$0$};
   \path (A2_1) edge [->] node [auto] {$\scriptstyle{}$} (A2_2);
    \path (A1_2) edge [->] node [auto,swap] {$\scriptstyle{}$} (A2_2);

    \end{tikzpicture}
\]
because $X$ is stable we can take the homotopy fiber of such
\[
 \begin{tikzpicture}
    \def\x{1.5}
    \def\y{-1.2}
    \node (A2_2) at (2*\x, 2*\y) {$A$};
    \node (A2_1) at (2*\x, 1*\y) {$0$};
    \node (A1_2) at (1*\x, 2*\y) {$0$};
   \node (A1_1) at (1*\x, 1*\y) {$\Omega A$};
   \path (A2_1) edge [->] node [auto] {$\scriptstyle{}$} (A2_2);
    \path (A1_2) edge [->] node [auto,swap] {$\scriptstyle{}$} (A2_2);
       \path (A1_1) edge [->] node [auto] {$\scriptstyle{}$} (A2_1);
   \path (A1_1) edge [->] node [auto] {$\scriptstyle{}$} (A1_2);

    \end{tikzpicture}
\]
Dually, we can consider the diagram
\[
 \begin{tikzpicture}
    \def\x{1.5}
    \def\y{-1.2}
    \node (A2_1) at (2*\x, 1*\y) {$0$};
    \node (A1_2) at (1*\x, 2*\y) {$0$};
   \node (A1_1) at (1*\x, 1*\y) {$A$};
       \path (A1_1) edge [->] node [auto] {$\scriptstyle{}$} (A2_1);
   \path (A1_1) edge [->] node [auto] {$\scriptstyle{}$} (A1_2);
    \end{tikzpicture}
\]
and take its homotopy cofiber
\[
 \begin{tikzpicture}
    \def\x{1.5}
    \def\y{-1.2}
    \node (A2_2) at (2*\x, 2*\y) {$\Sigma A$};
    \node (A2_1) at (2*\x, 1*\y) {$0$};
    \node (A1_2) at (1*\x, 2*\y) {$0$};
   \node (A1_1) at (1*\x, 1*\y) {$X$};
   \path (A2_1) edge [->] node [auto] {$\scriptstyle{}$} (A2_2);
    \path (A1_2) edge [->] node [auto,swap] {$\scriptstyle{}$} (A2_2);
       \path (A1_1) edge [->] node [auto] {$\scriptstyle{}$} (A2_1);
   \path (A1_1) edge [->] node [auto] {$\scriptstyle{}$} (A1_2);
    \end{tikzpicture}
\]
\newpage
Those construction allow to define $\infty$-functors
\[
\Sigma ,\Omega : X\to X 
\]
that are equivalences of $\infty$-categories \cite{Lurie 2}. Moreover, the axioms of stable $\infty$-category allow to define a notion of exact triangle in the homotopy category $h(X)$. Namely, a diagram in $h(X)$
\[
 \begin{tikzpicture}
    \def\x{1.5}
    \def\y{-1.2}
    \node (A1_1) at (2*\x, 2*\y) {$A$};
    \node (A2_2) at (3*\x, 2*\y) {$B$};
    \node (A3_3) at (4*\x, 2*\y) {$C$};
   \node (A4_4) at (5.5*\x, 2*\y) {$A[1]=\Sigma A$};
   \path (A1_1) edge [->] node [auto] {$\scriptstyle{f}$} (A2_2);
    \path (A2_2) edge [->] node [auto] {$\scriptstyle{g}$} (A3_3);
       \path (A3_3) edge [->] node [auto] {$\scriptstyle{h}$} (A4_4);
    \end{tikzpicture}
\]
 is a distinguished triangle if there exists a diagram $\Delta^1 \times \Delta^2 \to X$ of the type
 \[
 \begin{tikzpicture}
    \def\x{1.5}
    \def\y{-1.2}
    \node (A2_2) at (2*\x, 2*\y) {$D$};
    \node (A2_1) at (2*\x, 1*\y) {$0$};
    \node (A1_2) at (1*\x, 2*\y) {$C$};
   \node (A1_1) at (1*\x, 1*\y) {$B$};
  \node (A0_2) at (0*\x, 2*\y) {$0'$};
   \node (A0_1) at (0*\x, 1*\y) {$A$};
   \path (A2_1) edge [->] node [auto] {$\scriptstyle{}$} (A2_2);
   \path (A1_2) edge [->] node [auto,swap] {$\scriptstyle{\tilde{h}}$} (A2_2);
   \path (A1_1) edge [->] node [auto] {$\scriptstyle{}$} (A2_1);
   \path (A1_1) edge [->] node [auto] {$\scriptstyle{\tilde{g}}$} (A1_2);
      \path (A0_2) edge [->] node [auto] {$\scriptstyle{}$} (A1_2);
   \path (A0_1) edge [->] node [auto] {$\scriptstyle{\tilde{f}}$} (A1_1);
   \path (A0_1) edge [->] node [auto] {$\scriptstyle{}$} (A0_2);
    \end{tikzpicture}
\]
Such that
\begin{itemize}
\item 0 and 0' are both zero objects
\item Both squares are pushout diagrams in $X$
\item The morphisms $\tilde{f}$ an $\tilde{g}$ represent $f$ and $g$ respectively
\item The map $h$ is the composition with the homotopy class of $\tilde{h}$ with an equivalence $D\simeq A[1]$.
\end{itemize} 
We have the following proposition \cite{Lurie 2} 
\begin{prop}
Let $X$ be a stable $\infty$-category, then its homotopy category $h(X)$ with the shift functor
\[
\Sigma : h(X) \to h(X)
\]
and the class of distinguished triangles described in the previous remark is a triangulated category in the sense of Verdier \normalfont \cite{Ver}.
\end{prop}
\end{rem}
\newpage
\subsection{Homotopy limits and model theoretic properties of the Dold-Kan correspondence}
Let $\mathcal{C}$ be a small category and $I$ any (index) category. There is a functor
\[
(-)_{\ast} : \mathcal{C} \to Fun(I,\mathcal{C})
\] 
taking an object $x\in \mathcal{C}$ to the constant functor with value $x$. We say that $\mathcal{C}$ has $I$-shaped limits if the functor $(-)_{\ast}$ has  a right adjoint
\[
lim : Fun(I,\mathcal{C}) \to \mathcal{C}
\]
Dually, $\mathcal{C}$ has $I$-shaped colimits if $(-)_{\ast}$ has a left adjoint
\[
colim : Fun(I,\mathcal{C}) \to \mathcal{C}
\]
Assume that $\mathcal{C}$ has a structure of model category. We want to have a model structure of the category of functors $Fun(I,\mathcal{C})$ in a way that the limit-colimit adjunctions become Quillen adjunctions. In this situation, we can consider the total derived functors of the adjunction and define the homotopy limit of an $I$-shaped diagram as
\[
holim(X)=\mathbb{R}(lim)(X)=lim(P(X))
\]
where $P(X)$ is the fibrant replacement for the functor $X$ in the model structure on $Fun(I,\mathcal{C})$. A model structure with those properties on the functor category $Fun(I,\mathcal{C})$ always exists if the indexing category $I$ is Reedy \cite{Hir}. 
A Reedy category is a category together with a choice of two subcategories $\overrightarrow{I}$ and $\overleftarrow{I}$ called the direct subcategory and the inverse subcategory satisfying certain axioms. Once those subcategories are fixed, if $X\in Fun(I,\mathcal{C})$ is any functor and $i\in Ob(I)$ is any object, we can define the latching object of $X$ at $i$ 
\[
L_i(X)=\underset{d(\overrightarrow{I} \downarrow i)}{colim} X
\]
where $d(\overrightarrow{I} \downarrow i)$ is the over subcategory at $i$ in $\overrightarrow{I}$ with the identity at $i$ removed. Dually, the matching object at $i$ is defined
\[
M_i(X)=\underset{d(i \uparrow \overleftarrow{I})}{lim} X
\]
where $d(i \uparrow \overleftarrow{I})$ is the under subcategory at $i$ in $\overrightarrow{I}$ with the identity at $i$ removed. Now, given a morphism of functors $f:X\to Y$, the relative latching morphism at $i$ is defined as
\[
L_i(f): X_i \underset{L_i(X)}{\coprod} L_i(Y) \to Y_i
\]
and similarly the matching morphism at $i$
\[
M_i(f): X_i \to Y_i \underset{M_i(Y)}{\times} M_i(X)
\]
\newpage
\begin{defi}
Let $\mathcal{C}$ be a model category and $I$ an indexing Reedy category. The Reedy model structure on $Fun(I,\mathcal{C})$ is the model structure for which a morphism $f:X\to Y$ of functors is:
\begin{itemize}
\item A weak equivalence if it is an objectwise weak equivalence in $\mathcal{C}$.
\item A cofibration if the relative latching morphism is a cofibration in $\mathcal{C}$ for every object $i\in I$.
\item A fibration if the relative matching morphism is a fibration in $\mathcal{C}$, for every object $i\in I$.
\end{itemize}
\end{defi}
\begin{ex}\normalfont
We describe the Reedy model structure when $I$ is the category
\[
  \begin{tikzpicture}
    \def\x{1.5}
   \def\y{-1.2}
    \node (A2_2) at (2*\x, 2*\y) {$0$};
    \node (A2_1) at (2*\x, 1*\y) {$1$};
    \node (A1_2) at (1*\x, 2*\y) {$2$};
   \path (A2_1) edge [->] node [auto] {$\scriptstyle{}$} (A2_2);
    \path (A1_2) edge [->] node [auto,swap] {$\scriptstyle{}$} (A2_2);
    \end{tikzpicture}
  \]
to compute homotopy fibre products. We set $\overleftarrow{I}=I$ and $\overrightarrow{I}$ to be subcategory with the same objects as $I$ and morphisms given by just the identities. Then for a functor $X\in Fun(I,\mathcal{C})$,
\[
M_0(X)=\ast_{in}
\]
where $\ast_{in}$ is the initial object of $\mathcal{C}$ and that
\[
M_i(X)=X_0
\]
for $i=1,2$. The latching objects are all equal to the terminal object $\ast_{ter}$. If $f:X\to Y$ is a morphism of functors, then the relative matching morphisms are given by
\[
M_0(f): X_0 \to Y_0
\] 
and for $i=1,2$
\[
M_i(f): X_i \to Y_i \underset{Y_0}{\times}X_0
\]
The relative latching morphism are $L_i(f)=f_i$ for $i=0,1,2$. In particular a morphism $f$ is a weak equivalence (cofibration) if it is an objectwise weak equivalence (cofibration) and fibrations are morphisms for which $X_i\to M_i(f)$ are fibrations. In particular, cofibrant diagrams are diagrams in which every object is cofibrant and fibrant diagrams are diagrams for which $X_0$ is a fibrant object and the morphisms $X_1\to X_0$, $X_2\to X_0$ are fibrations. In particular, given a diagram
\[
 \begin{tikzpicture}
    \def\x{1.5}
    \def\y{-1.2}
    \node (A2_2) at (2*\x, 2*\y) {$X_1$};
    \node (A2_1) at (2*\x, 1*\y) {$X_0$};
    \node (A1_2) at (1*\x, 2*\y) {$X_2$};
   \path (A2_1) edge [->] node [auto] {$\scriptstyle{f_{01}}$} (A2_2);
    \path (A1_2) edge [->] node [auto,swap] {$\scriptstyle{f_{21}}$} (A2_2);
    \end{tikzpicture}
\]
then a fibrant replacement of it is the diagram
\[
 \begin{tikzpicture}
    \def\x{1.5}
    \def\y{-1.2}
    \node (A2_2) at (2*\x, 2*\y) {$P(X_1)$};
    \node (A2_1) at (2*\x, 1*\y) {$P(X_0)$};
    \node (A1_2) at (0.8*\x, 2*\y) {$P(X_2)$};
   \path (A2_1) edge [->] node [auto] {$\scriptstyle{\tilde{f}_{01}}$} (A2_2);
    \path (A1_2) edge [->] node [auto,swap] {$\scriptstyle{\tilde{f}_{21}}$} (A2_2);
    \end{tikzpicture}
\]
where $P$ is the fibrant replacement on $\mathcal{C}$. If $X_1$ is already a fibrant object of $\mathcal{C}$, then a fibrant replacement is given by 
\[
\begin{tikzpicture}
    \def\x{1.5}
    \def\y{-1.2}
    \node (A2_2) at (2*\x, 2*\y) {$X_1$};
    \node (A2_1) at (2*\x, 1*\y) {$P(f_{01})$};
    \node (A1_2) at (1*\x, 2*\y) {$P(f_{21})$};
   \path (A2_1) edge [->>] node [auto] {$\scriptstyle{}$} (A2_2);
    \path (A1_2) edge [->>] node [auto,swap] {$\scriptstyle{}$} (A2_2);
    \end{tikzpicture}
\]
where 
\[
  \begin{tikzpicture}
    \def\x{1.5}
    \def\y{-1.2}
    \node (A1_1) at (0.5*\x, 2*\y) {$X_i$};
    \node (A2_2) at (2*\x, 2*\y) {$X_1$};
    \node (A1_2) at (1.25*\x, 1*\y) {$P(f_{i1})$};
   \path (A1_1) edge [right hook->] node [above] {$\scriptstyle{\simeq}$} (A1_2);
    \path (A1_2) edge [->>] node [auto] {$\scriptstyle{}$} (A2_2);
    \path (A1_1) edge [->] node [auto] {$\scriptstyle{f_{i1}}$} (A2_2);
     \end{tikzpicture}
  \]
is a trivial cofibration-fibration factorization of the morphism $f_{i1}$, $i=0,2$. Hence, if $X_1$ is a fibrant object in $\mathcal{C}$, then
\[
X_2 \underset{X_1} {\overset{h} \times} X_0 \simeq  P(f_{21}) \underset{X_1} {\times} P(f_{01})
\]
\end{ex}
\begin{ex}\normalfont
Let $\mathcal{C}=Ch_{\bullet}^{\ge0}(Mod_{\mathbb{K}})$. This category has a model structure, called the projective model structure. Weak equivalences are quasi-isomorphisms of chain complexes, fibrations are chain maps which are surjective in positive degree and cofibrations are monomorphims of chain complexes with projective cokernel in every degree. In this model structure, every chain complex is fibrant. Moreover, given a chain map $f_{\bullet}: A_{\bullet} \to B_{\bullet}$, a trivial cofibration-fibration factorization of this map is given by
\[
  \begin{tikzpicture}
    \def\x{1.5}
    \def\y{-1.2}
    \node (A1_1) at (0.5*\x, 2*\y) {$A_{\bullet}$};
    \node (A2_2) at (2*\x, 2*\y) {$B_{\bullet}$};
    \node (A1_2) at (1.25*\x, 1*\y) {$P(f_{\bullet})$};
   \path (A1_1) edge [right hook->] node [above] {$\scriptstyle{i    \simeq}$} (A1_2);
    \path (A1_2) edge [->>] node [auto] {$\scriptstyle{p}$} (A2_2);
    \path (A1_1) edge [->] node [auto] {$\scriptstyle{f_{\bullet}}$} (A2_2);
     \end{tikzpicture}
  \]
where $P(f_{\bullet})$ is the chain complex
\[
P(f_{\bullet})_n=A_n\bigoplus B_{n+1}\bigoplus B_n
\]
for $n>0$ and
\[
P(f_{\bullet})_0=A_0\bigoplus B_{1}\bigoplus D_0
\]
where $D_0\subseteq B_0$ given by equations $b_0=d(b_1)+f_0(a_0)$ for $b_1\in B_1$ and $a_0\in A_0$. The differential is given by
\[
d_n=\begin{vmatrix} d_{A_n} & 0 & 0 \\ -f_n & -d_{B_{n+1}} & Id_{B_n} \\ 0 & 0 & d_{B_n} \end{vmatrix}
\]
and the morphisms $i$ and $p$ by
\[
i_n=\begin{vmatrix} Id_{A_n} & 0 & f_n \end{vmatrix}
\]
\[
p_n=\begin{vmatrix} 0 \\ 0 \\ Id_{B_n} \end{vmatrix}
\]
One can easily check that $H_{\ast}(P(f_{\bullet}))\simeq H_{\ast}(A_{\bullet})$ and that $p$ is a fibration, being degreewise surjective. The fact that $i$ is a cofibration follows from the fact that every $\mathbb{K}$-module is free when $\mathbb{K}$ is a field and hence projective. We compute homotopy fibre products in the model category $Ch_{\bullet}^{\ge0}(Mod_{\mathbb{K}})$. Let $X\in Fun(I,Ch_{\bullet}^{\ge0}(Mod_{\mathbb{K}}))$
\[
 \begin{tikzpicture}
    \def\x{1.5}
    \def\y{-1.2}
    \node (A2_2) at (2*\x, 2*\y) {$X^0_{\bullet}$};
    \node (A2_1) at (2*\x, 1*\y) {$X^1_{\bullet}$};
    \node (A1_2) at (1*\x, 2*\y) {$X^2_{\bullet}$};
   \path (A2_1) edge [->] node [auto] {$\scriptstyle{f^{01}}$} (A2_2);
    \path (A1_2) edge [->] node [auto,swap] {$\scriptstyle{f^{02}}$} (A2_2);
    \end{tikzpicture}
\]
Then a fibrant replacement of this diagram is given by 
\[
 \begin{tikzpicture}
    \def\x{1.5}
    \def\y{-1.2}
    \node (A2_2) at (2*\x, 2*\y) {$X^0_{\bullet}$};
    \node (A2_1) at (2*\x, 1*\y) {$P(f^{01})$};
    \node (A1_2) at (1*\x, 2*\y) {$P(f^{02})$};
   \path (A2_1) edge [->] node [auto] {$\scriptstyle{p^{01}}$} (A2_2);
    \path (A1_2) edge [->] node [auto,swap] {$\scriptstyle{p^{02}}$} (A2_2);
    \end{tikzpicture}
\]
In particular the homotopy limit is computed as 
\[
X^1_{\bullet} \underset{X^0_{\bullet}} {\overset{h} \times} X^2_{\bullet} \simeq  P(f^{01}) \underset{X^0_{\bullet}} {\times} P(f^{02})
\]
For instance, the loop object of a chain complex 
\[
\Omega(X_{\bullet})=0 \underset{X_{\bullet}} {\overset{h} \times} 0
\] 
is given by the limit of the diagram
\[
 \begin{tikzpicture}
    \def\x{1.5}
    \def\y{-1.2}
    \node (A2_2) at (2*\x, 2*\y) {$X^0_{\bullet}$};
    \node (A2_1) at (2*\x, 1*\y) {$P(0)$};
    \node (A1_2) at (1*\x, 2*\y) {$P(0)$};
   \path (A2_1) edge [->] node [auto] {$\scriptstyle{p}$} (A2_2);
    \path (A1_2) edge [->] node [auto,swap] {$\scriptstyle{p}$} (A2_2);
    \end{tikzpicture}
\]
In particular we find that 
\[
P(0) \underset{X_{\bullet}} {\times} P(0)_n \simeq X_{n+1}
\]
for $n>0$ and 
\[
P(0) \underset{X_{\bullet}} {\times} P(0)_0 \simeq Ker(d_{X_1})
\]
hence $\Omega(X_{\bullet})\simeq X[-1]$.
\end{ex}
\newpage
\begin{rem}\normalfont
Recall that the category $S(Mod_{\mathbb{K}})$ has a model structure, induced by the forgetful functor to the model category of simplicial sets with the Quillen model structure \cite{Hov}. Weak equivalences are maps of simplical $\mathbb{K}$-modules whose induced map on the underlying simplicial sets is a weak equivalences, fibrations are maps whose induced map on the underlying simplicial sets is a Kan fibration. Cofibrations are retract of maps satisfying the right lifting property with respect to trivial fibrations. According to S.Schwede-B.Shipley \cite{Schw}, the functors of the Dold-Kan correspondence are both left and right adjoint of a Quillen adjunction between $S(Mod_{\mathbb{K}})$ with the described model structure and $Ch_{\bullet}^{\ge0}(Vect_{\mathbb{K}})$ with the projective model structure. This can be rephrased saying that the computation of homotopy limits in the model category $S(Vect_{\mathbb{K}})$ is equivalent, under the Dold-Kan correspondence, to the computation of homotopy limits in the model category $Ch_{\bullet}^{\ge0}(Vect_{\mathbb{K}})$. Namely, let
\[
N: S(Vect_{\mathbb{K}}) \to Ch_{\bullet}^{\ge0}(Vect_{\mathbb{K}})   
\]
the normalized chain complex functor of the Dold-Kan correspondence and 
\[
N_{\ast}: Fun(I,S(Vect_{\mathbb{K}})) \to Fun(I,Ch_{\bullet}^{\ge0}(Vect_{\mathbb{K}}))  
\]
the induced functor on the diagram categories. Then for $Y\in Fun(I,Ch_{\bullet}^{\ge0}(Vect_{\mathbb{K}}))$ we have
\[
N_{\ast}(P(Y))=P(N_{\ast}(Y))
\]
where $P$ is the fibrant replacement. This because by \cite{Schw}, $N$ is the right and left adjoint of a Quillen equivalence. In particular, it preserves fibrations and trivial cofibrations and hence it preserves factorizations. Hence, for $Y\in Fun(I,S(Vect_{\mathbb{K}}))$, we have
\[
holim(\mathbb{R}N_{\ast}(Y))=lim(P(N_{\ast}(P(Y))))=lim(P(N_{\ast}(Y)))
\]
and hence 
\[
holim(Y)\simeq \mathbb{L}DK(lim(P(N_{\ast}(Y))))\simeq DK(Q(lim(P(N_{\ast}(Y)))))\simeq DK(lim(P(N_{\ast}(Y))))
\]
because every chain complex is cofibrant in the projective model structure. In particular, if $Y=DK_{\ast}(X)$, for some $X\in Fun(I,Ch_{\bullet}^{\ge0}(Vect_{\mathbb{K}}))$ then we find that
\[
holim(Y)\simeq DK(lim(P(X)))
\]
\end{rem}
\newpage
\subsection{Pretriangulated dg-categories give stable $\infty$-categories}
We prove now that given a pretriangulated dg-category $\mathcal{D}$, $N^{big}_{dg}(\mathcal{D})$, and hence $N^{sm}_{dg}(\mathcal{D})$, are stable $\infty$-categories.
\begin{thm}
Let $\mathcal{D}$ be a pretriangulated dg-category, then the big nerve $N^{big}_{dg}(\mathcal{D})$ is a stable $\infty$-category. Moreover, $H^0(\mathcal{D})$ is equivalent to $h(N^{big}_{dg}(\mathcal{D}))$ as triangulated categories.
\end{thm}
\begin{proof}\normalfont
By definition, a pretriangulated dg-category has a $0$ object, meaning that is both final and terminal satisfying for every object $X$ of $\mathcal{D}$
\[
Hom^{\bullet}_{\mathcal{D}}(X,0)\simeq 0 \simeq Hom^{\bullet}_{\mathcal{D}}(0,X)
\]
hence, by construction, 
\[
Map_{\mathcal{D}_{\Delta}}(X,0)\simeq \ast \simeq Map_{\mathcal{D}_{\Delta}}(0,X)
\] 
Hence, $0$ is both initial and terminal in $N^{big}_{dg}(\mathcal{D})$. We have to prove that given a $1$-simplex $f:X \to Y $ in $N^{big}_{dg}(\mathcal{D})$, i.e. a closed degree $0$ morphism in $\mathcal{D}$, it has fiber and cofiber, namely the diagrams
\[
 \begin{tikzpicture}
    \def\x{1.5}
    \def\y{-1.2}
    \node (A2_2) at (2*\x, 2*\y) {$Y$};
    \node (A2_1) at (2*\x, 1*\y) {$X$};
    \node (A1_2) at (1*\x, 2*\y) {$0$};
   \path (A2_1) edge [->] node [auto] {$\scriptstyle{f}$} (A2_2);
    \path (A1_2) edge [->] node [auto,swap] {$\scriptstyle{0}$} (A2_2);
    \end{tikzpicture}
\]
\[
 \begin{tikzpicture}
    \def\x{1.5}
    \def\y{-1.2}
    \node (A2_2) at (1*\x, 1*\y) {$X$};
    \node (A2_1) at (2*\x, 1*\y) {$Y$};
    \node (A1_2) at (1*\x, 2*\y) {$0$};
   \path (A2_2) edge [->] node [auto] {$\scriptstyle{f}$} (A2_1);
    \path (A2_2) edge [->] node [auto,swap] {$\scriptstyle{0}$} (A1_2);
    \end{tikzpicture}
\]
admit respectively homotopy limits and colimits. We prove that there exists homotopy cartesian 
\[
 \begin{tikzpicture}
    \def\x{1.5}
    \def\y{-1.2}
    \node (A2_2) at (3*\x, 3*\y) {$Y$};
    \node (A2_1) at (3*\x, 1*\y) {$X$};
    \node (A1_2) at (1*\x, 3*\y) {$0$};
   \node (A1_1) at (1*\x, 1*\y) {$Cone(f)[-1]$};
   \path (A2_1) edge [->] node [auto] {$\scriptstyle{f}$} (A2_2);
    \path (A1_2) edge [->] node [auto,swap] {$\scriptstyle{0}$} (A2_2);
       \path (A1_1) edge [->] node [auto] {$\scriptstyle{i}$} (A2_1);
   \path (A1_1) edge [->] node [auto] {$\scriptstyle{0}$} (A1_2);
   \end{tikzpicture}
\]
and cocartesian diagrams
\[
 \begin{tikzpicture}
    \def\x{1.5}
    \def\y{-1.2}
    \node (A2_2) at (1.5*\x, 1*\y) {$X$};
    \node (A2_1) at (3.5*\x, 1*\y) {$Y$};
    \node (A1_2) at (1.5*\x, 3*\y) {$0$};
       \node (A1_1) at (3.5*\x, 3*\y) {$Cone(f)$};
  \path (A2_1) edge [->] node [auto] {$\scriptstyle{j}$} (A1_1);
   \path (A1_2) edge [->] node [auto] {$\scriptstyle{0}$} (A1_1);
   \path (A2_2) edge [->] node [auto] {$\scriptstyle{f}$} (A2_1);
    \path (A2_2) edge [->] node [auto,swap] {$\scriptstyle{0}$} (A1_2);
    \end{tikzpicture}
\]
Consider the case for the cofiber of $f$. Let $j:Y\to Cone (f)$ be the $1$-simplex corresponding to degree 0 morphism
\[
Hom^0_{\mathcal{D}}(Y,Cone(f))=Hom^0_{\mathcal{D}}(Y,Y)\oplus Hom^1_{\mathcal{D}}(Y,X)
\]
given by $(Id_{Y},0)$. This is a closed morphism because
\[
d(Id_{Y},0)=(d(Id_Y)+f\circ 0, d(0))=(0,0)
\]
and the composition $j\circ f$ is null homotopic in the sense that if we take 
\[
h=(0,Id_Y)\in Hom^{-1}_{\mathcal{D}}(Y,Cone(f))=Hom^{-1}_{\mathcal{D}}(Y,Y)\oplus Hom^0_{\mathcal{D}}(Y,X)
\] 
then 
\[
d(h)=(d(0)+f\circ Id_Y, d(0))=(f,0)=j\circ f
\] 
Now we prove that for every $Z\in Ob(\mathcal{D})$, the induced maps of simplicial sets
\[
Map_{C_{\Delta}}(Cone(f),Z)\to Map_{C_{\Delta}}(Y,Z)\underset{Map_{C_{\Delta}}(X,Z)}{\overset{h}{\times}} \ast
\]
are weak equivalences. The homotopy limit of the right hand side can be computed via the Dold-Kan correspondence. First, we compute what is the homotopy fibre product of the diagram of chain complexes
\[
 \begin{tikzpicture}
    \def\x{4}
    \def\y{-1.5}
    \node (A2_2) at (2*\x, 2*\y) {$\tau_{\ge0}(Hom^{\bullet}_{\mathcal{D}}(X,Z)^{op})$};
    \node (A2_1) at (2*\x, 1*\y) {$0$};
    \node (A1_2) at (0.7*\x, 2*\y) {$\tau_{\ge0}(Hom^{\bullet}_{\mathcal{D}}(Y,Z)^{op})$};
   \path (A2_1) edge [->] node [auto] {$\scriptstyle{-\circ f}$} (A2_2);
    \path (A1_2) edge [->] node [auto,swap] {$\scriptstyle{0}$} (A2_2);
    \end{tikzpicture}
\]
According to computation of homotopy fibre products in model categories, the homotopy fibre product of this diagram is the fibre product of the diagram
\[
 \begin{tikzpicture}
    \def\x{4}
    \def\y{-1.5}
    \node (A2_2) at (2*\x, 2*\y) {$\tau_{\ge0}(Hom^{\bullet}_{\mathcal{D}}(X,Z)^{op})$};
    \node (A2_1) at (2*\x, 1*\y) {$P(0)$};
    \node (A1_2) at (0.8*\x, 2*\y) {$P(-\circ f)$};
   \path (A2_1) edge [->>] node [auto] {$\scriptstyle{p_1}$} (A2_2);
    \path (A1_2) edge [->>] node [auto,swap] {$\scriptstyle{p_2}$} (A2_2);
    \end{tikzpicture}
\]
the chain complex $P(-\circ f)$ is given by
\[
\begin{cases} 
P(-\circ f)_0=Ker(d_{Hom^0_{\mathcal{D}}(Y,Z)})\oplus Hom^{-1}_{\mathcal{D}}(X,Z) \oplus D^0, &\mbox{if } k = 0 \\
P(-\circ f)_k=Hom^{-k}_{\mathcal{D}}(Y,Z)\oplus Hom^{-k-1}_{\mathcal{D}}(X,Z) \oplus Hom^{-k}_{\mathcal{D}}(X,Z), &\mbox{if } k > 0. 
\end{cases} 
\]
where $D^0\subseteq Hom^{0}_{\mathcal{D}}(X,Z)$ is given by the equation 
\[
g^0=h^0\circ f + d(g^{-1})
\] 
for $h^0\in Ker(d_{Hom^{0}_{\mathcal{D}}(Y,Z)})$ and $g^{-1}\in Hom^{-1}_{\mathcal{D}}(X,Z)$. 
\newpage
Differential is given for $k=0$ by
\[
d_0=\begin{vmatrix} 
d^{0}_{(Y,Z)}& 0& 0\\ 
-(- \circ f)& -d^{-1}_{(X,Z)}& Id_{(X,Z)}\\ 
0& 0& d^{0}_{(Y,Z)} 
\end{vmatrix}
\]
and for $k>0$
\[
d_k=\begin{vmatrix} d^{-k}_{(Y,Z)} & 0 & 0 \\ -(- \circ f) & -d^{-k-1}_{(X,Z)} & Id_{(X,Z)} \\ 0 & 0 & d^{-k}_{(Y,Z)} \end{vmatrix}
\]
where $d^{k}_{(X,Y)}=d_{Hom^{k}_{\mathcal{D}}(X,Y)}$ for $X,Y\in \mathcal{D}$.
Similarly $P(0)$ is given by
\[
\begin{cases} 
P(0)_0=Hom^{-1}_{\mathcal{D}}(X,Z) \oplus Im(d_{Hom^{-1}_{\mathcal{D}}(X,Z)}), &\mbox{if } k = 0 \\
P(0)_k=Hom^{-k-1}_{\mathcal{D}}(X,Z) \oplus Hom^{-k}_{\mathcal{D}}(X,Z), &\mbox{if } k > 0. 
\end{cases} 
\]
with differential for $k=0$ given by
\[
d_0=\begin{vmatrix} -d^{-1}_{(X,Z)} & Id_{(X,Z)} \\ 0 & d^{0}_{(Y,Z)} \end{vmatrix}
\]
and for $k>0$
\[
d_k=\begin{vmatrix} -d^{-k-1}_{(X,Z)} & Id_{(X,Z)} \\ 0 & d^{-k}_{(Y,Z)} \end{vmatrix}
\]
Hence, the homotopy fiber product is given for $k>0$ by
\[
Hom^{-k}_{\mathcal{D}}(Y,Z)\oplus Hom^{-k-1}_{\mathcal{D}}(X,Z)\oplus Hom^{-k-1}_{\mathcal{D}}(X,Z) \oplus Hom^{-k}_{\mathcal{D}}(X,Z)
\]
with differential 
\[
d_k=\begin{vmatrix} d^{-k}_{(X,Z)} & 0 & 0 & 0 \\ -(- \circ f) & -d^{-k-1}_{(X,Z)} & 0 & 0 \\ 0 & 0 & -d^{-k-1}_{(X,Z)} & Id_{(X,Z)} \\ 0 & 0 & 0 & d^{-k}_{(X,Z)} \end{vmatrix}
\]
and for $k=0$ by the subvector space of the direct sum $P(- \circ f)_0 \oplus P(0)_0$ given by the identification of the components $D^0$ and $Im(d_{Hom^{-1}_{\mathcal{D}}(X,Z)})$ via the inclusion in $Hom^{0}_{\mathcal{D}}(X,Z)$. 
This chain complex, using the quasi-isomorphism $P(0)\simeq 0$, is the chain complex whose degree $k>0$ piece is given by
\[
Hom^{-k}_{\mathcal{D}}(Y,Z)\oplus Hom^{-k-1}_{\mathcal{D}}(X,Z)
\]
with differential 
\[
d_k=\begin{vmatrix} d^{-k}_{(X,Z)} & 0  \\ -(- \circ f) & -d^{-k-1}_{(X,Z)} \end{vmatrix}
\]
\newpage
and in degree $0$ by
\[
Ker(d_{Hom^{0}_{\mathcal{D}}(Y,Z)})\oplus E^{-1} 
\]
where $E^{-1}\subseteq Hom^{-1}_{\mathcal{D}}(X,Z)$ is defined by the equation 
\[
d(g^{-1}) + g^0\circ f=0
\]
for some $g^0\in Ker(d_{Hom^{0}_{\mathcal{D}}(Y,Z)})$. This chain complex is clearly quasi isomorphic to $\tau_{\ge0}(Hom^{\bullet}_{\mathcal{D}}(Cone(f),Z)^{op})$ with induced morphism 
\[
\tau_{\ge0}(Hom^{\bullet}_{\mathcal{D}}(Cone(f),Z)^{op}) \to \tau_{\ge0}(Hom^{\bullet}_{\mathcal{D}}(Y,Z)^{op})\underset{\tau_{\ge0}(Hom^{\bullet}_{\mathcal{D}}(X,Z)^{op})}{\overset{h}{\times}} 0
\] 
being a quasi-isomorphims. A similar proof holds in the case of homotopy fibers. In this case there is a morphism $i: Cone(f)[-1] \to X$ inducing, for every $Z\in Ob(\mathcal{D})$ an homotopy equivalence
\[
\tau_{\ge0}(Hom^{\bullet}_{\mathcal{D}}(Z, Cone(f)[-1])^{op}) \to \tau_{\ge0}(Hom^{\bullet}_{\mathcal{D}}(Z,X)^{op})\underset{\tau_{\ge0}(Hom^{\bullet}_{\mathcal{D}}(Z,Y)^{op})}{\overset{h}{\times}} 0
\] 
Such morphism is given by 
\[
i=(Id_X,0) \in Hom^{0}_{\mathcal{D}}(X,X)\oplus Hom^{1}_{\mathcal{D}}(Y,X)=Hom^{0}_{\mathcal{D}}(Cone(f)[-1], X)
\] 
Moreover, computation of the homotopy limit of the diagram
\[
 \begin{tikzpicture}
    \def\x{4}
    \def\y{-1.5}
    \node (A2_2) at (2*\x, 2*\y) {$\tau_{\ge0}(Hom^{\bullet}_{\mathcal{D}}(Z,Y)^{op})$};
    \node (A2_1) at (2*\x, 1*\y) {$\tau_{\ge0}(Hom^{\bullet}_{\mathcal{D}}(Z,X)^{op})$};
    \node (A1_2) at (1*\x, 2*\y) {$0$};
   \path (A2_1) edge [->] node [auto] {$\scriptstyle{-\circ f}$} (A2_2);
    \path (A1_2) edge [->] node [auto,swap] {$\scriptstyle{0}$} (A2_2);
    \end{tikzpicture}
\]
gives the chain complex 
\[
\begin{cases} 
Ker(d_{Hom^{0}_{\mathcal{D}}(Z,X)})\oplus F^{-1}, &\mbox{if } k = 0 \\
Hom^{-k}_{\mathcal{D}}(Z,X)\oplus Hom^{-k-1}_{\mathcal{D}}(Z,Y), &\mbox{if } k > 0. 
\end{cases} 
\]
where $F^{-1}\subseteq Hom^{-1}_{\mathcal{D}}(Z,Y)$ given by the equation 
\[
d(g^{-1}) + g^0\circ f=0
\] 
for some $g^0\in Ker(d_{Hom^{0}_{\mathcal{D}}(Z,X)})$ and differential given for $k>0$ by
\[
d_k=\begin{vmatrix} d^{-k}_{(Z,X)} & 0  \\ -(- \circ f) & -d^{-k-1}_{(Z,Y)} \end{vmatrix}
\]
This chain complex is quasi-isomorphic via the morphism induced by $i$ to the chain complex $\tau_{\ge0}(Hom^{\bullet}_{\mathcal{D}}(Z, Cone(f)[-1])^{op})$.
\newpage
Let $f: X \to Y$ be a morphism and consider the induced morphism $j: Y \to Cone(f)$. We prove that there is an homotopy equivalence $X\simeq Cone(j)[-1]$ in $N^{big}_{dg}(\mathcal{D})$. In particular this implies that cocartesian diagrams are also cartesian. Recall that $Cone(j)[-1]$ is the twisted complex given by 
\[
 \begin{tikzpicture}
    \def\x{3}
    \def\y{-1.5}
    \node (A2_2) at (2*\x, 2*\y) {$(Y)_1$};
    \node (A2_1) at (1*\x, 2*\y) {$(Y\oplus X)_0$};
   \path (A2_1) edge [->] node [auto] {$\scriptstyle{Id_{Y}\oplus f}$} (A2_2);
    \end{tikzpicture}
\]
We have that  
\[
Hom^0_{\mathcal{D}}(X,Cone(j)[-1])=Hom^{-1}_{\mathcal{D}}(X,Y)\oplus Hom^0_{\mathcal{D}}(X,Y\oplus X)
\]
and 
\[
Hom^0_{\mathcal{D}}(Cone(j)[-1],X)=Hom^1_{\mathcal{D}}(Y,X)\oplus Hom^0_{\mathcal{D}}(Y\oplus X,X)
\]
moreover we have closed degree $0$ morphisms $\pi_{X}: Y\oplus X \to X$, $i_X: X \to Y\oplus X$ such that $\pi_{X}\circ i_{X}=Id_{X}$. Define
\begin{align*}
&h=(0,\pi_X)\\
&g=(0,(-f)\oplus i_X) \, 
\end{align*} 
it is easy to show that those are closed degree $0$ morphisms. Clearly, on one hand, $h\circ g=Id_{X}$. On the other hand, there exists 
\[
\alpha \in Hom^{-1}_{\mathcal{D}}(Cone(j)[-1],Cone(j)[-1])
\] 
such that $d(\alpha)=g\circ h - Id_{Cone(j)[-1]}$. Recall that $Hom^k_{\mathcal{D}}(Cone(j)[-1],Cone(j)[-1])$ is the cochain complex
\[
Hom^k_{\mathcal{D}}(Y,Y)\oplus Hom^{k+1}_{\mathcal{D}}(Y,Y\oplus X)\oplus Hom^{k-1}_{\mathcal{D}}(Y\oplus X,Y)\oplus Hom^k_{\mathcal{D}}(Y\oplus X,Y\oplus X) 
\]
in particular, the differential in degree $-1$ is given by 
\[
d_{-1}(c^{-1},c^0,c^{-2},d^{-1})=\begin{vmatrix} d(c^{-1})+(Id_{Y}\oplus f)\circ c^0 \\ d(c^{0}) \\ d(c^{-2})+c^{-1}\circ (Id_{Y}\oplus f)+(Id_{Y}\oplus f)\circ d^{-1} \\ d(d^{-1})+c^0\circ (Id_{Y}\oplus f)   \end{vmatrix}
\]
Now, we have that
\[
g\circ h - Id_{Cone(j)[-1]}=(-Id_{Y},0,0,-Id_{Y\oplus X}+((-f)\oplus i_x)\circ \pi_X)
\]
and we have the identities
\begin{align*}
&(Id_{Y}\oplus f)\circ (-i_{Y})=-Id_{Y} \\
&(-i_Y)\circ (Id_{Y}\oplus f)= -Id_{Y\oplus X}+((-f)\oplus i_x)\circ \pi_X \, 
\end{align*} 
hence taking
\[
\alpha=(0,-i_Y,0,0)
\]
it is clear that $d(\alpha)=g\circ h - Id_{Cone(j)[-1]}$. A similar argument holds in order to prove that cartesian diagrams are also cocartesian. The fact that $H^0(\mathcal{D})$ is equivalent to $h(N^{big}_{dg}(\mathcal{D}))$ as triangulated categories follows easily from the proof.
\end{proof}

\raggedright
\bibliographystyle{plainnat}


\end{document}